\documentclass[12pt]{amsart}
\usepackage{amssymb,amsmath}
\def\R {\mathbb{R}}

\def\H {{\mathcal H}}

\def \and{\qquad\text{and}\qquad}

\def\Dt{\partial_t}
\def\Dx{\Delta_x}
\def\Nx{\nabla_x}
\def\Cal{\mathcal}
\def\eb{\varepsilon}
\def\({\left(}
\def\){\right)}
\def\Bbb{\mathbb}
\def\dist{\operatorname{dist}}

\newtheorem{proposition}{Proposition}[section]
\newtheorem{theorem}[proposition]{Theorem}
\newtheorem{corollary}[proposition]{Corollary}
\newtheorem{lemma}[proposition]{Lemma}
\theoremstyle{definition}
\newtheorem{definition}[proposition]{Definition}
\newtheorem{remark}[proposition]{Remark}

\numberwithin{equation}{section}

\def \no#1#2#3 {{\bf #1} (#3), #2.}
\def \eds#1#2#3 {#1, #2, #3.}

\title[quasi-linear strongly damped wave equation]{Finite-dimensional attractors
for the quasi-linear strongly-damped wave equation}

\author[Varga Kalantarov and Sergey Zelik]{ Varga Kalantarov${}^1$ and Sergey Zelik${}^2$}

\begin{document}
\begin{abstract} We present a new method of investigating  the
so-called quasi-linear strongly damped wave equations
$$
\Dt^2u-\gamma\Dt\Dx u-\Dx u+f(u)=\Nx\cdot \phi'(\Nx u)+g
$$
in  bounded 3D domains. This method allows us to establish the
existence and uniqueness of energy solutions in the case where the
growth exponent of the non-linearity $\phi$ is less than 6 and $f$
may have arbitrary polynomial growth rate. Moreover, the existence of
a finite-dimensional global and exponential attractors for the
solution semigroup associated with that equation and their
additional regularity are also established. In a particular case
$\phi\equiv0$ which corresponds to the so-called semi-linear
strongly damped wave equation, our result allows to remove the
long-standing growth restriction $|f(u)|\le C(1+ |u|^5)$.
\end{abstract}

\begin{address}
 {${}^1$ Department of Mathematics, Ko\c c University, Istanbul,
 Turkey}
\end{address}
\begin{address}
 {${}^2$  Department of Mathematics, University of Surrey, Guildford, GU2 7XH, UK}
\end{address}


\maketitle
\section{Introduction}\label{i}
We consider the following quasi-linear strongly damped wave equation
in a smooth bounded domain $\Omega\subset\R^3$:
\begin{equation}
\begin{cases}\label{1.main}
\Dt^2u-\gamma\Dt\Dx u-\Dx u+f(u)=\Nx\cdot \phi'(\Nx u)+g,\\
u\big|_{\partial\Omega}=0, \ \ \xi_u(0):=(u(0),\Dt u(0))=\xi_0,
\end{cases}
\end{equation}
where $\Dx $ is a Laplacian with respect to the variable
$x\in\Omega$, $u=u(t,x)$ is an unknown function, $\gamma>0$ is a fixed positive number, $g\in
L^2(\Omega)$ are given external forces, and  $\phi\in C^2(\R^3\rightarrow
\R^1) $ and $f\in C^2(\R^1\rightarrow \R^1)$ are known functions satisfying the following
conditions:
\begin{equation}\label{1.int}
 a_0|\eta|^{p-1}\le\phi''(\eta)\le a_1(1+|\eta|^{p-1}), \ \ \ \ \forall
\eta \in\R^3,
\end{equation}
for some positive $a_i$ and $p\in[1,5)$
and
\begin{equation}\label{3.int}
-C+ a|s|^{q} \leq f'(s)\leq C(1+|s|^{q} ),\ \ \ \ \
\forall s \in\R^1,
\end{equation}
for some  $C>0$, $a>0$, and $q>0$.
Note that assumption \eqref{1.int} reads
\begin{equation}\label{2.int}
a_0 |\eta|^{p-1}|\xi|^2\le
\sum_{i,j=1}^3\phi''_{\eta_i\eta_j}(\eta)\xi_i\xi_j \le  a_1(1+ |\eta|^{p-1})|\xi|^2, \ \ \ \
\ \forall \xi, \eta \in\R^3
\end{equation}
and typical examples of the nonlinearities satisfying above
conditions  \eqref{1.int} and \eqref{3.int} are
\begin{equation}\label{3.nonlinearity}
 f(u)= u|u|^q-Cu,  \ \ \ \ \ \phi (\eta)=
|\eta|^{p+1}.
\end{equation}

Equations of the form \eqref{1.main} occur in the study of motion of
viscoelastic materials. For instance, in one dimensional case, they model
longitudinal vibration of a uniform, homogeneous bar with nonlinear
stress law given by the function $\phi(u_x)$. In two  and
three dimensional cases they describe antiplane shear motions
of viscoelastic solids. We refer to  \cite{FiLaOn,Kn,MM} for physical origins
and derivation of mathematical models of motions of viscoelastic media and only recall here
that, in applications, the unknown $u$ naturally  represents the displacement of the body relative
to a fixed reference configuration. By this reason, it would be
more physical to consider equation \eqref{1.main} in the {\it vector} case $u=(u^1,u^2,u^3)$.
 However, in order to avoid the additional technicalities,
we prefer to study only the {\it scalar} case of equation
\eqref{1.main} although the most part of our
 results can be straightforwardly extended to the vector case with
 the {\it convex} energy $\phi$ satisfying the analogue of
 conditions \eqref{1.int}.
 It  also worth to note that, even for $n=1$,  problem
\eqref{1.main}  may not have a global classical solution if the viscosity term $\gamma\Dt\Dx u$
is omitted. Thus the inclusion of
 this damping  term represents a regularization of the equation and prevents the blow up of
 solutions.
\par
The Cauchy problem (for the case $\Omega=\R^n$) and the boundary
value problem for equation \eqref{1.main} under the different assumptions on the nonlinearities $\phi$ and $f$
have been studied in many papers (see
\cite{Cl1,Cl2,En1,En2,En3,Fei,FrNe,GGZ,He,IMM,KPS,KS,Le,MM,NN,On,Pe,Ya1} and
references therein). The most understood is the semi-linear case
$\phi\equiv0$:
\begin{equation}\label{sem1}
\Dt^2u-\gamma\Dt\Dx u-\Dx u+f(u)=g
\end{equation}
 which is  usually refereed as strongly damped wave equation or
pseudohyperbolic equation, see e.g.,
\cite{AS,GM,He,PZ,K86,K88,Le,We}.
As well as we know, one of the first essential results on global
behavior of solutions  problem  \eqref{sem1} was  obtained
by Webb \cite{We}. He proved that, if $\Omega \in\R^3$ and the
nonlinear term satisfies the standard dissipativity conditions
(without any growth restrictions),
 problem \eqref{sem1} has a unique strong solution belonging
to the space $[H^2\cap H^1_0]\times L^2$  and each strong solution
 tends to the appropriate stationary solution  as $t\rightarrow \infty$. This result inspired the further
studies of the long-time behavior of solutions of that equations.
 In particular, the related operator-differential equation in
a Banach space $X$:
\begin{equation}\label{sem2}
u_{tt}+\alpha Au_t +Au =f(u), \ \alpha >0
\end{equation}
was considered in \cite{Fi,Ma}.
Here $A$ is a sectorial operator with compact resolvent on $X$ and $f$ is a nonlinear
operator satisfying some regularity and growth conditions.
Then the semigroup
 $S(t)$ generated by the Cauchy problem for \eqref{sem2}
 in the phase space  $X^{\sigma}\times X^{\beta}, 0\leq \sigma \leq \beta<1$, $X^s:=D(A^{s/2})$ is well
 posed, dissipative and asymptotically compact and therefore possesses
a global attractor in the above phase space.

Semi-linear equation \eqref{sem1} with the growth restriction
$f(u)\sim u|u|^{q}$, $q<4$ is studied in \cite{K86,K88,GM}.
It was shown that the  semigroup associated with energy solutions possesses  a global
attractor whose fractal dimension is finite. The closest to our
study are the results and methods of \cite{PZ} where the regularity of the global attractor
in the energy phase space for three
dimensional problem \eqref{sem1} with nonlinear term growing as
$u|u|^4$  is  established, see also \cite{Pata}. Note that, during the long time,
the growth exponent $q=4$ has been considered as a critical one
for establishing the  uniqueness and asymptotic regularity of
energy solutions for \eqref{sem1}. As we will see below, this
restriction is occurred artificial and the above result remains true
without any  restrictions on the exponent $q$.
\par
The quasi-linear case $\phi\ne0$ is much more complicated. However, in the case of one spatial dimension,
more or less complete theory of this equation can be found in the
literature
(see
\cite{GMM},\cite{CGW}, \cite{DG}, \cite{KLY} and references
therein).
 The study of the analytic properties and long-time behavior of
 solutions of such quasi-linear problems was initiated by the
 paper of Greenberg, MacCamy and Misel \cite{GMM}, where the
authors considered initial boundary value problem for the equation
$$
\rho \sb{0}u\sb{tt}=\phi \sp{\prime} \,(u\sb{x})u\sb{xx}+\lambda
u\sb{xtx}.
$$
  It is shown that if $\lambda>0$, and the initial
functions are sufficiently smooth, then  a unique, smooth solution
exists, which moreover goes to zero as $t$ goes to infinity.

The existence of a global attractor for the one-dimensional problem
$$
u_{tt}-\alpha u_{xxt}-\partial_x\phi(u_x)+f(u)=g(x),\  x\in(0,1), \ t>0
$$
in the phase space $[H^2\cap H^1_0]\times H^1_0$ is established in
\cite{CGW} (under natural conditions on nonlinear terms, see also \cite{Ber}) and the
existence of an exponential attractor for that problem in the  space
$[H^3 \cap H^1_0] \times H^1_0$ can be found in \cite{DG}. However, even in that relatively
 simple 1D case, the well-posedness of the problem in the
class of {\it energy solutions} was problematic (if $\phi''$ is not
globally bounded).
\par
The multi-dimensional case is essentially less understood. Some
exception is only the case where the non-linearity
$\phi$ has globally bounded second derivative $\|\phi''(\eta)\|\le C$ (and the non-linearity
$f$ satisfies the standard growth restrictions). The typical example of such
non-linearity  is
\begin{equation}\label{n.kor}
\phi(\eta):=\sqrt{1+|\eta|^2}.
\end{equation}
At that case the global existence and uniqueness is
straightforward for any space dimension even in the class of
energy solutions. Moreover, rewriting equation \eqref{1.main} in
the form of
\begin{equation}\label{0.funny}
\Dt u(t)-\gamma\Dx u(t)=\Dt u(0)-\gamma\Dx u(0)+\int_0^t(\Nx\cdot\phi'(\Nx
u(s))+\Dx u(s)-f(u(s))+g)\,ds
\end{equation}
and treating it as a perturbation of a heat equation, one obtains
the control of a $W^{2,s}$-norm of the solution $u$ for some $s>n$ which is
sufficient to verify the global existence of classical (and even
more regular) solutions, see \cite{Cl2,IMM,KPS,Ry,Tv} and references
therein. Note however, that the trick with the
integro-differential equation \eqref{0.funny} works {\it only} for
globally bounded $\phi''$ and gives the estimates which grow
exponentially (or even doubly exponentially) in time which is not
very helpful for the attractor theory.
\par
In contrast to that, the well-posedness of problem \eqref{1.main}   for the case of {\it
growing} non-linearities $\phi$ (satisfying \eqref{1.int} with
$p>1$) has been previously known only for the two-dimensional
case. For instance, the case $\Omega=\R^2$ with the non-linearity
$$
\phi(\eta):=\Phi_1(\eta_1)+\Phi_2(\eta_2),
$$
where the functions $\phi_i:=\Phi_i'$ have at most cubic growth
is considered in \cite{Pe}, the case of spherically symmetric
nonlinearity
$$
\phi(\eta)=\tilde \phi(|\eta|^2)
$$
with an arbitrary polynomial  growth (in a bounded domain $\Omega\subset\R^2$)
is studied in \cite{En3} and the case of spherically symmetric
solutions in any space dimension is investigated in \cite{En1,En2}.
But again, the proof of uniqueness used there requires the $C^1$-regularity
of solutions whose  existence is based on the
analogue of the integro-differential equation \eqref{0.funny} and,
by this reason, the associated a priori estimates are
divergent in time.
\par
It also worth to mention here the theory of so-called small solutions
(which is usual for quasi-linear hyperbolic equations). In particular,
Kawashima and Shibata \cite{KS} proved global existence and
 stability  of {\it small} smooth solutions of the initial boundary
value problem for a quasilinear system of strongly damped wave
equations of the form
$$
A_0(U)u_{tt}+\sum_{j=1}^n
A_j(U)u_{x_jt}-\sum_{ij=1}^n A_{ij}(U)u_{x_ix_j}=\sum_{ij=1}^n
B_{ij}(U)u_{x_ix_jt}, \ \ U=(\nabla u,u_t,\nabla u_t)
$$
in any space dimension ($A_0(U)$, $A_j(U)$ and $A_{ij}(U)$ are assumed to satisfy some
natural assumptions near $U=0$; see also \cite{E}), the analogous results for less
regular solutions of the
equation
\begin{equation}\label{0.ex}
u_{tt}-\nabla_x\cdot\left(\frac{\nabla_x u}{\sqrt{1+|\nabla_x
u|^2}}\right)-\Delta u_t-|u|^{\alpha}u=0,
\end{equation}
where $\alpha $ is a given positive number are established in \cite{IMM}.
\par
Thus, to the best of our knowledge, the methods developed in the
previous papers were {\it insufficient} for developing the attractor
theory for the quasi-linear equation \eqref{1.main} in the
multi-dimensional case and, by that reason, this problem has not
been considered in the literature. However, a lot of related results
for the simplified versions of equation \eqref{1.main} can be
found:  we would like to mention here the papers
\cite{BHJPS,On,PB} devoted to global behavior of solutions of
the so-called Kirchhoff equation with strongly damping term (an equation of the form \eqref{1.main} with
$\phi(\Nx u):=\Phi(\|\Nx u\|_{L^2}^2)$) and the papers \cite{BGS,Pi} devoted to the
global behavior of solutions  for the
following operator differential  equation in a  Hilbert space
$\H$:
\begin{equation}\label{abstr2}
u_{tt}+A_1 u +A_2 u_t+ N^{*}g(Nu)=h(t),
\end{equation}
where $A_1,A_2$ and $N$ are given linear operators and $g(\cdot)$ is a
 nonlinear  operator. Note that, alhough our equation \eqref{1.main} can be considered as a particular case of
  \eqref{abstr2} with $A_1=A_2:=-\Dx$ and $N:=\Nx$, that choice of
  operators does not satisfy the assumptions of the aforementioned
  papers (in particular, the operator $A_2^{-1/2}N$ is not
  compact).
\par
In the present paper, we give a comprehensive study of equation
\eqref{1.main} in a 3D bounded domain with smooth boundary including
the well-posedness of weak energy solutions, asymptotic compactness, existence and finite-dimensionality
of global attractors, further regularity of solutions, etc.  In
particular, we present here a new method of verifying the
uniqueness of energy solutions which allows us not only to treat
the general case of equation \eqref{1.main}, but also essentially improve the
theory even in the semi-linear case $\phi\equiv0$ by removing the
long-standing growth restriction for the non-linearity $f$. In
order to avoid the additional technicalities, we pay the main
attention to the most complicated case of growing non-linearities $\phi$ (with $p\ge1$) and,
by this reason, we have written the additional
term $\Dx u$ in equation \eqref{1.main} which automatically
excludes the case of degenerate equations as well as the case of
the non-linearity \eqref{n.kor} and equation \eqref{0.ex} (the
most part of our results can be easily extended to the case of
equation \eqref{0.ex}, but the additional accuracy related with
the "non-coercivity" of the non-linearity \eqref{n.kor} is
required).

The paper is organized as follows.
\par
The well-posedness and dissipativity of weak energy solutions are
studied in Section \ref{s1}. We recall that the energy functional
for equation \eqref{1.main} reads
\begin{equation}\label{0.enf}
E(u,\Dt u):=\frac12\|\Dt u\|^2_{L^2}+\frac12\|\Nx
u\|^2_{L^2}+(\phi(\Nx u),1)+(F(u),1)-(g,u),
\end{equation}
where $F(v):=\int_0^vf(z)\,dz$ and $(\cdot,\cdot)$ stands for the
inner product in $L^2(\Omega)$ and, therefore, the natural choice
of the energy phase space is the following one:
\begin{equation}\label{0.ens}
\Cal E:=[W^{1,p+1}_0(\Omega)\cap L^{q+2}(\Omega)]\times L^2(\Omega).
\end{equation}
(here and below, we denote by $W^{s,p}(\Omega)$ the Sobolev space
of distributions whose derivative up to order $s$ belong to
$L^p(\Omega)$) and,
by definition, a weak energy solution of equation \eqref{1.main}
is a function
\begin{equation}\label{1.spaces}
u\in L^\infty([0,T],W^{1,p+1}_0(\Omega)\cap L^{q+2}(\Omega))\cap
W^{1,\infty}([0,T],L^2(\Omega))\cap W^{1,2}([0,T],W^{1,2}(\Omega))
\end{equation}
which satisfyies the equation in the sense of distributions. We note
that, as usual, the trajectory $t\to\xi_u(t):=(u(t),\Dt u(t))$ is
weakly continuous with respect to $t$ as the $\Cal E$-valued
function.
By this reason,  the initial data at $t=0$ is well-defined.
\par
The main result of this section is the following theorem.
\begin{theorem}\label{TH0.wp} Let the non-linearities $\phi$ and $f$ satisfy assumptions \eqref{1.int}
and \eqref{3.int} respectively. Then, for any $\xi_u(0)\in\Cal E$,
problem \eqref{1.main} has a unique weak energy solution
$\xi_u(t)$ and this solution satisfies the dissipative estimate:
\begin{equation}\label{0.dis}
\|\xi_u(t)\|_{\Cal E}^2+\int_t^{t+1}\|\Nx\Dt u(s)\|^2_{L^2}\,ds\le
Q(\|\xi_u(0)\|_{\Cal E}^2)e^{-\alpha t}+Q(\|g\|_{L^2})
\end{equation}
for some positive constant $\alpha$ and monotone function $Q$
independent of $u$ and $t$.
\end{theorem}
In a slight abuse of notations, we
denote by $\|\xi_u(t)\|_{\Cal E}$ the following energy "norm" in
the space $\Cal E$:
$$
\|\xi_u(t)\|_{\Cal E}^2:=\|\Dt u(t)\|^2_{L^2}+\|\Nx
u(t)\|^{p+1}_{L^{p+1}}+\|u(t)\|^{q+2}_{L^{q+2}}.
$$
Moreover, a number of important additional results, including the
Lipschitz continuity with respect to the initial data (in a
slightly weaker norm) and a partial smoothing property for the
$\Dt u$-component of the solution $\xi_u(t)$ are obtained there.
\par
The existence of a global attractor $\Cal A$ for the semigroup
$S(t)$ associated with weak energy solutions of \eqref{1.main} is
verified in Section \ref{s.atr} using the so-called method of
$l$-trajectories. The  attractor $\Cal A$ attracts the images of bounded
subsets in $\Cal E$ only in the topology of slightly weaker space
$\tilde{\Cal E}:=W^{1,2}_0(\Omega)\times L^2(\Omega)$ and has
finite fractal dimension in $\tilde{\Cal E}$. This drawback is
partially corrected below (in Section \ref{s4}) where the
attraction property in the initial phase space is verified for
the semi-linear case $\phi\equiv0$ (and arbitrarily growing
non-linearity $f$).
\par
The so-called {\it strong} solutions of problem \eqref{1.main} are
considered in Section \ref{s2}. By definition, that solutions
belong to the space
\begin{equation}\label{1.e1energy1}
\Cal E_1:=\{u\in W^{2,2}(\Omega)\cap W^{1,2}_0(\Omega),\, \Dt u\in
W^{1,2}_0(\Omega), \ \Nx(\phi'(\Nx u))\in W^{-1,2}(\Omega)\},
\end{equation}
for every $t\ge0$ (see Remark \ref{Rem3.space} below for more information
 about the structure of this space) and the "norm" in that space
 is naturally defined by
$$
\|\xi_u(t)\|_{\Cal E_1}^2:=\|\Dt
u(t)\|^2_{H^1}+\|u(t)\|_{W^{2,2}}^2+\|\Nx(\phi'(\Nx
u(t))\|^2_{W^{-1,2}}.
$$
The main result of this section is the
following theorem.
\begin{theorem}\label{Th0.s} Let the assumptions of Theorem
\ref{TH0.wp} hold and let, in addition, $\xi_u(0)\in\Cal E_1$.
Then the associated weak energy solution of \eqref{1.main} is, in
a fact, a strong solution and the following estimate holds:
\begin{equation}\label{0.sdis}
\|\xi_u(t)\|_{\Cal E_1}+\int_t^{t+1}\|\Dt^2 u(s)\|^2_{L^2}\,ds\le
Q(\|\xi_u(0)\|_{\Cal E_1})e^{-\alpha t}+Q(\|g\|_{L^2})
\end{equation}
for some positive constant $\alpha$ and monotone function $Q$
independent of $t$.
\end{theorem}
The proof of that theorem is based on a combination of our new
"uniqueness technique" (which is now used for obtaining the proper
estimates for $\Dt u$) with the
technique of \cite{En2} extended to the spherically non-symmetric
case (which allows us to estimate the $W^{2,2}$-norm of the
solution). As usual, that $W^{2,2}$-estimate is based on the
multiplication of the equation \eqref{1.main} by $\Dx u$ (and
works without any growth restrictions on $p$), however, in
contrast to \cite{En2}, we would prefer to use some kind of
non-linear localization technique rather than direct estimates of
the boundary terms arising after the integration by parts in the
term $(\Nx\cdot\phi'(\Nx u),\Dx u)$, see Lemma \ref{Lem3.h2} for
details.
\par
In addition, the well-posedness and dissipativity of problem
\eqref{1.main} for the {\it critical} growth exponent $p=4$ for
the non-linearity $\phi$ in the class of {\it strong} solutions
are verified in that section.
\par
However, we have to emphasize that the strong solutions are a
priori not regular enough in order to prevent possible
singularities of the gradient $\Nx u$ and to prove that $\Nx u\in
C$, so, the global existence and dissipativity of {\it classical}
(and more regular) solutions remains an open problem here.
\par
The asymptotic regularity of weak energy solutions and the
existence of an {\it exponential} attractor of the semigroup
$S(t)$ associated with weak energy solutions are verified in
Section \ref{s4}. In particular, using the extension of the
technique, developed in \cite{PZ} to the quasi-linear case,
we establish that the proper ball $\Bbb B$ in the space $\Cal E_1$ is an
exponentially attracting set for the semigroup $S(t)$, i.e., that,
for any bounded set $B\subset \Bbb B$,
\begin{equation}\label{0.exp}
\dist_{\tilde{\Cal E}}(S(t)B,\Bbb B)\le
 Q(\|B\|_{\Cal E})e^{-\alpha t},
\end{equation}
where the positive constant $\alpha$ and monotone function $Q$ are
independent of $B$.
\par
This result immediately implies that the above mentioned global
attractor $\Cal A$ consists of {\it strong} solutions and together
with the results of Section \ref{s.atr} gives the existence of an
exponential attractor $\Cal M$ for the solution semigroup.
\par
In addition, in the particular semilinear case $\phi\equiv0$, we
prove the analogue of estimate \eqref{0.exp} for the initial phase
space $\Cal E$ which gives the attraction property in the topology
of the initial phase space as well. Thus, we have constructed the
complete theory for the semi-linear case with arbitrary polynomial
growth rate of the non-linearity $f$.
\par
Finally, in the concluding Section \ref{s6}, we discuss the
applications of our technique to the related, but more simple
equations, including the so-called Kirchhof equation, membrane
equation, and the semilinear wave equation with structural damping.
\par
It is also worth to note that, although only the case of a bounded
domain $\Omega$ is considered in the paper, the result can be
extended to the case of unbounded domains $\Omega$. We will return
to that question somewhere else.

\section{Energy solutions: existence, uniqueness and
dissipativity}\label{s1}

In this section, we start to study  weak energy solutions of problem
\eqref{1.main}. The following standard theorem gives the solvability
and dissipativity of the problem \eqref{1.main}.
\begin{theorem}\label{Th1.dis} Let the nonlinearities $\phi$ and $f$ satisfy  assumptions \eqref{1.int} and
\eqref{3.int}, $g\in L^2(\Omega)$  and $\xi_u(0)\in\Cal E$. Then, problem \eqref{1.main} possesses at
least one weak energy solution $\xi_u(t)$ which satisfies the
following dissipative estimate:
\begin{equation}\label{1.energy}
\|\xi_u(t)\|_{\Cal E}^2+\int_t^{t+1}\|\Dt\Nx u(s)\|_{L^2}^2\,ds\le
Q(\|\xi_u(0)\|_{\Cal E}^2)e^{-\beta t}+Q(\|g\|_{L^2}),
\end{equation}
where $\beta$ is a positive constant and $Q$ is a monotone function
both independent of the initial data $\xi_u(0)$.
\end{theorem}
\begin{proof} Since the result of the theorem is more or less standard,
we restrict ourselves by only formal derivation of
estimate \eqref{1.energy} which can be easily justified using,
e.g., the Galerkin approximation method and the monotonicity
arguments for passing to the limit in the quasi-linear term, see \cite{BV} for details.
\par
Indeed, multiplying equation \eqref{1.main} by $\Dt u$ and
integrating over $x\in\Omega$, we arrive at
\begin{equation}\label{1.ener}
\frac d{dt}\(\frac12 \|\Dt u\|^2_{L^2}+(\phi(\Nx
u),1)+\frac12\|\Nx u\|^2_{L^2}+(F(u),1)-(g,u)\)+\gamma\|\Nx\Dt
u\|^2_{L^2}=0.
\end{equation}
Here $F(u):=\int_0^uf(v)\,dv$ and $(\cdot,\cdot)$ stands for the
usual inner product in $L^2(\Omega)$.
\par
Multiplying  equation \eqref{1.main} by $\alpha u$, where
$\alpha$ is a small positive number which will be fixed below,
we obtain
\begin{multline}\label{1.en1}
\frac d{dt}(\alpha\frac\gamma2\|\Nx u\|^2_{L^2}+\alpha(\Dt
u,u))-\alpha\|\Dt u\|^2+\\+\alpha\|\Nx u\|^2_{L^2}+\alpha(\phi'(\Nx
u),\Nx u)+\alpha(f(u),u)=\alpha(g,u).
\end{multline}
Taking a sum of that two equations and denoting
$$
E(\xi_u(t))\!:=\!\frac12 \|\Dt u\|^2_{L^2}+(\phi(\Nx
u),1)+\frac12\|\Nx u\|^2_{L^2}+(F(u),1)-(g,u)+\alpha\frac\gamma2\|\Nx u\|^2_{L^2}+\alpha(\Dt
u,u),
$$
we deduce that
$$
\frac d{dt}E(\xi_u(t))+\gamma\|\Nx\Dt
u\|^2_{L^2}-\alpha\|\Dt u\|^2_{L^2}+\alpha\|\Nx u\|^2_{L^2}+\alpha(\phi'(\Nx
u),\Nx u)+\alpha(f(u),u)=\alpha(g,u).
$$
We  fix $\alpha$ be small enough that
\begin{multline}\label{1.esten}
\beta(\|\Dt u\|^2_{L^2}+\|\Nx
u\|^{p+1}_{L^{p+1}}+\|u\|^{q+1}_{L^{q+1}})-C(1+\|g\|^2_{L^2})\le\\\le
E(\xi_u(t))\le C(1+\|g\|^2_{L^2}+\|\Dt u\|^2_{L^2}+\|\Nx
u\|^{p+1}_{L^{p+1}}+\|u\|^{q+1}_{L^{q+1}})
\end{multline}
and that $\|\Dt\Nx u\|^2_{L^2}\ge 2\alpha\|\Dt u\|^2_{L^2}$ (it is possible to do due
 to our assumptions on $\phi$ and $f$, see \eqref{1.int} and \eqref{3.int}). Then,
the last inequality gives
$$
\frac d{dt}E(\xi_u(t))+\delta E(\xi_u(t))+\delta\|\Dt\Nx u\|^2_{L^2}\le C(1+\|g\|^2_{L^2})
$$
for some positive constants $\delta$ and $C$. Applying the
Gronwall inequality to this estimate and using \eqref{1.esten}, we arrive at the
desired estimate \eqref{1.energy} and finish the proof of the
theorem.
\end{proof}
Note that the growth restriction $p<5$ is nowhere used in the
proof of Theorem \ref{Th1.dis} and, therefore, that result holds
for arbitrary growth $p$. In contrast to that, the next uniqueness
result essentially uses this growth restriction.

\begin{theorem}\label{Th1.uniq} Let the assumptions of Theorem
\ref{Th1.dis} hold. Then, the energy solution of problem
\eqref{1.main} is unique. Moreover, for every two energy solutions
$u_1(t)$ and $u_2(t)$
(with different initial data), the following Lipschitz continuity
in a weaker space holds:
\begin{equation}\label{1.dif}
\|\Dt v(t)\|^2_{H^{-1}}+\|v(t)\|^2_{H^1}\le C e^{Kt}(\|\Dt
v(0)\|^2_{H^{-1}}+\|v(0)\|^2_{H^1}),
\end{equation}
where $v(t):=u_1(t)-u_2(t)$ and the constants $C$ and $K$ depend
only on the $\Cal E$-norms of the initial data.
\end{theorem}
\begin{proof}
Indeed, the function $v(t)$ solves
\begin{equation}\label{1.difer}
\begin{cases}
\Dt^2 v-\gamma\Dt \Dx v-\Dx v+[f(u_1)-f(u_2)]=\Nx(\phi'(\Nx
u_1)-\phi'(\Nx u_2)),\\
v\big|_{\partial\Omega}=0, \ \ v(0)=u_1(0)-u_2(0), \ \Dt v(0))=\Dt
u_1(0)-\Dt u_2(0).
\end{cases}
\end{equation}
In order to estimate the non-linear terms in that equation, we
need the following lemma.
\begin{lemma}\label{Da} (see for instance \cite{Da} or \cite{Z1}) If the function $\phi$ satisfies conditions
\eqref{1.int}  then there exist a constant
$\delta>0 $ such that
\begin{equation}\label{4.int}
[\phi'(\eta_1)-\phi'(\eta_2)]\cdot [\eta_1-\eta_2] \geq
\delta\left(|\eta_1|+|\eta_2|\right)^{p-1}|\eta_1-\eta_2|^2, \ \forall
\eta_1,\eta_2 \in \R^3.
\end{equation}
\end{lemma}
Thus, due to our assumptions on $f$ and Lemma \ref{Da},
$$
[f(u_1)-f(u_2)].(u_1-u_2)\ge -C|v|^2+d_0(|u_1|^{q}+|u_2|^{q})|v|^2,
$$
for some positive $d_0$.
Analogously, for the function $\phi'$, we have
$$
\left(\phi'(\Nx u_1)-\phi'(\Nx u_2),\Nx v\right)\geq
d_0((|u_1|+|u_2|)^{p-1},|v|^2).
$$
Multiplying  equation \eqref{1.difer} by $v$, integrating over $x\in\Omega$ (it is not difficult to check that
all of the integrals have sense; as usual, the regularity of the energy solution is sufficient for multiplication
of the equation by $u$, but not for $\Dt u$) and using the above estimates for
the functions $f$ and $\phi$, we will have
\begin{multline}\label{1.est}
\Dt\left[\frac{\gamma}2\|\Nx v\|_{L^2}^2+(v,\Dt
v)\right]+d_0((|\Nx u_1|+|\Nx u_2|)^{p-1}+1,|\Nx
v|^2)+\\+d_0((|u_1|+|u_2|)^q,|v|^2)\le C\|v\|_{L^2}^2+\|\Dt
v\|^2_{L^2}
\end{multline}
for some positive $d_0$.
\par
In order to control the term with $\Dt v$ in the right-hand side
of \eqref{1.est}, we multiply equation \eqref{1.difer} by
$(-\Dx)^{-1}\Dt v$ and integrate over $x\in\Omega$ (again, all
 of the integrals will have sense). Then, we have
\begin{multline}\label{1.est1}
\Dt(\|\Dt v\|^2_{H^{-1}}+\|v\|_{L^2}^2)+2\gamma \|\Dt
v\|^2_{L^2}\le\\\le 2(|f(u_1)-f(u_2)|,|(-\Dx)^{-1}\Dt v|)+
2(|\phi'(\Nx u_1)-\phi'(\Nx u_2)|,|\Nx(-\Dx)^{-1}\Dt v|).
\end{multline}
Using the growth assumptions on $f$ and the embedding
$W^{2,2-s}(\Omega)\subset C(\Omega) ,s\in [0,\frac12)$, the first
term in the right-hand side of \eqref{1.est1} can be estimated as
follows:
\begin{multline}\label{1.fest}
 (|f(u_1)-f(u_2)|,|(-\Dx)^{-1}\Dt
 v|)\le\eb\|f(u_1)-f(u_2)\|_{L^1}^2+C_\eb\|\Dt
 v\|_{H^{-s}}^2\le\\\le
 \eb((|u_1|+|u_2|)^q,|v|)^2+\eb\|\Dt v\|_{L^2}^2+C_\eb\|\Dt
 v\|_{H^{-1}}^2\le\\\le
 \eb(\|u_1|^{q}+|u_2|^{q},1)(|u_1|^q+|u_2|^q,|v|^2)+\eb\|\Dt
 v\|^2_{L^2}+C_\eb\|\Dt v\|_{H^{-1}}^2\le\\\le \eb(\|\xi_{u_1}\|_{\Cal
 E}+\|\xi_{u_2}\|_{\Cal E})^{q}((|u_1|+|u_2|)^q,|v|^2)+\eb\|\Dt
 v\|_{L^2}^2+C_\eb\|\Dt v\|_{H^{-1}}^2,
\end{multline}
 where $\eb>0$ is arbitrary.

Analogously, using
 the growth restriction $p+1<6$, together with the H\"older inequality and the interpolation
\begin{multline*}
\|\Nx(-\Dx)^{-1}\Dt v\|_{L^{p+1}}^2\le \eb\|\Nx(-\Dx)^{-1}\Dt
v\|^2_{H^1}+\\+C_\eb\|\Nx(-\Dx)^{-1}\Dt v\|_{L^2}^2=\eb\|\Dt
v\|^2_{L^2}+C_\eb\|\Dt v\|^2_{H^{-1}},
\end{multline*}
 one can estimate the second
term in the right-hand side of \eqref{1.est1} as follows:
\begin{multline}\label{1.phiest}
(|\phi'(\Nx u_1)-\phi'(\Nx u_2)|,|\Nx(-\Dx)^{-1}\Dt
v|)\le\\\le\eb\|\phi(\Nx u_1)-\phi(\Nx u_2)\|_{L^{(p+1)/p}}^2+\eb\|\Nx
(-\Dx)^{-1}\Dt v\|_{L^{p+1}}^2\le\\\le
\eb((|\Nx u_1|+|\Nx u_2|)^{(p^2-1)/p},|\Nx
v|^{(p+1)/p})^{2p/(p+1)}+\eb\|\Dt v\|^2_{L^2}+C_\eb\|\Dt
v\|^2_{H^{-1}}\le\\
\eb((|\Nx u_1|+|\Nx u_2|)^{p+1},1)^{(p-1)/(p+1)}((|\Nx
u_1|+|\Nx u_2|)^{p-1}, |\Nx v|^2)+\eb\|\Dt v\|^2_{L^2}+C_\eb\|\Dt
v\|_{H^{-1}}^2\\\le \eb(\|\xi_{u_1}\|_{\Cal E}+\|\xi_{u_2}\|_{\Cal
E})^{p-1}((|\Nx u_1|+|\Nx u_2|)^{p-1}, |\Nx v|^2)+\eb\|\Dt
v\|_{L^2}^2+C_\eb\|\Dt v\|^2_{H^{-1}},
\end{multline}
where $\eb>0$ is again arbitrary.
\par
Inserting estimates \eqref{1.fest} and \eqref{1.phiest} into the
right-hand side of \eqref{1.est1}, taking a sum with estimate
\eqref{1.est} multiplied by $\frac2\gamma$, using the energy estimate \eqref{1.energy} for
estimating the $\Cal E$-norms of $\xi_{u_i}$ and fixing $\eb>0$
small enough, we finally deduce that
\begin{multline}\label{1.dlip}
\Dt \left[\frac\gamma2\|\Nx v\|_{L^2}^2+(\Dt v,v)+\frac2\gamma(\|\Dt
v\|_{H^{-1}}^2+\|v\|^2_{L^2})\right]+\\+2\beta(\|\Dt v\|^2_{L^2}+\|\Nx v\|^2_{L^2})\le C(\|v\|^2_{L^2}+\|\Dt
v\|^2_{H^{-1}}),
\end{multline}
where the constant $C$ depends on the $\Cal E$-norms of
$\xi_{u_1}(0)$ and $\xi_{u_2}(0)$.
Moreover, obviously, the function
\begin{equation}\label{1.e-1}
E_{-1}(v):=\frac\gamma2\|\Nx v\|_{L^2}^2+(\Dt v,v)+\frac2\gamma(\|\Dt
v\|_{H^{-1}}^2+\|v\|^2_{L^2})
\end{equation}
satisfies
\begin{equation}\label{1.2sest}
\kappa_1(\|\Dt v\|_{H^{-1}}^2+\|\Nx v\|^2_{H^{1}})\le E_{-1}(v)\le
\kappa_2(\|\Dt v\|_{H^{-1}}^2+\|\Nx v\|^2_{H^{1}})
\end{equation}
for some positive $\kappa_i$.
Applying the Gronwall inequality to \eqref{1.dlip}, we see that
\begin{equation}\label{1.Lip}
\|\Dt v(t)\|_{H^{-1}}^2+\|v(t)\|^2_{H^1}\le Ce^{Kt}(\|\Dt
v(0)\|_{H^{-1}}^2+\|v(0)\|_{H^1}^2)
\end{equation}
and the uniqueness holds.
\end{proof}
Thus, for every $\xi_u(0)\in\Cal E$, there exists a unique weak
solution $\xi_u(t)$ of problem \eqref{1.main} and, therefore, the
solution semigroup
\begin{equation}\label{1.sem}
S(t):\Cal E\to\Cal E,\ \ S(t)\xi_u(0):=\xi_u(t),\ \ \ t\ge0
\end{equation}
is well-defined in the energy space $\Cal E$. Moreover, as the
proved theorem shows, this semigroup is Lipschitz continuous in
the weaker space $\Cal E_{-1}:=H^1_0(\Omega)\cap H^{-1}(\Omega)$:
\begin{equation}\label{1.lip}
\|S(t)\xi_1-S(t)\xi_2\|_{\Cal E_{-1}}\le
Ce^{Kt}\|\xi_1-\xi_2\|_{\Cal E_{-1}},
\end{equation}
where the constants $C$ and $K$ depend on the $\Cal E$-norms of
$\xi_1$ and $\xi_2$. We will essentially use below this weak
Lipschitz continuity in order to prove the finite-dimensionality
of the global attractor.
\par
Our next task is to establish some partial smoothing property
of the constructed semigroup $S(t)$. We start with the case where the
$\Dt u$-component of the initial data is more regular.

\begin{proposition}\label{Prop1.dtreg} Let the assumptions of
Theorem \ref{Th1.uniq} be satisfied  and let the initial data $\xi_u(0)$ be
such that
\begin{multline}\label{1.reg}
\Dt u(0)\in H^1_0(\Omega),\ \ \ \Dt^2 u(0):=\gamma\Dt\Dx u(0)+\\+\Dx
u(0)+\Nx\phi'(\Nx u(0))-f(u(0))+g\in H^{-1}(\Omega).
\end{multline}
Then, the function $v(t):=\Dt u(t)$ is such that $\xi_v(t)\in\Cal
E_{-1}$ for every $t\ge0$ and the following estimate holds:
\begin{equation}\label{1.dtest}
\|\xi_v(t)\|^2_{\Cal E_{-1}}+\int_t^{t+1}\|\Dt v(s)\|^2_{L^2}\,ds\le
Q(\|\xi_v(0)\|_{\Cal E_{-1}}+\|\xi_u(0)\|_{\Cal E})e^{-\gamma
t}+Q(\|g\|_{L^2}),
\end{equation}
where the positive constant $C$ and monotone function $Q$
are independent of $t$ and $u$.
\end{proposition}
\begin{proof} As in the proof of the Theorem \ref{Th1.dis}, we give below only
formal derivation of the estimate \eqref{1.dtest} which can be
justified, e.g., by the Galerkin approximation method.
\par
Indeed, the function $v(t)$ solves
\begin{equation}\label{1.dteq}
\Dt^2 v-\gamma\Dt\Dx v-\Dx v+f'(u)v=\Nx(\phi''(\Nx u)\Nx v),\ \
\xi_v(0)\in\Cal E_{-1}
\end{equation}
which is of the form \eqref{1.difer}. By this reason, multiplying
this equation by the function $v(t)+\frac2\gamma(-\Dx)^{-1}\Dt v(t)$ and
arguing exactly as in the derivation of \eqref{1.dlip}, we end up
with the following inequality:
\begin{multline}\label{1.dtdif}
\Dt E_{-1}(v(t))+\beta E_{-1}(v(t))+\\+\beta(\|\Dt v\|^2_{L^2}+
\|\Nx v\|^2_{L^2})\le Q(\|\xi_u(t)\|_{\Cal E})(\|v(t)\|_{L^2}^2+\|\Dt
v(t)\|_{H^{-1}}^2),
\end{multline}
where $\beta>0$ is a fixed constant, $Q$ is a given monotone
function and the functional $E_{-1}(v)$ is given by \eqref{1.e-1}.
Moreover, using the  embedding $L^1\subset H^{-2}$ and
expressing $\Dt^2 u$ from equation \eqref{1.main}, we have
\begin{equation}\label{1.dt2}
\|\Dt v(t)\|_{H^{-3}}\le Q(\|\xi_u(t)\|_{\Cal E})+\|g\|_{L^2}
\end{equation}
for some monotone function $Q$. Indeed, taking the inner product
of equation \eqref{1.e-1} with arbitrary test function $\varphi\in
H^3=D((-\Dx)^{3/2})$, we have
$$
(\Dt^2 u,\varphi)=\gamma(\Dt
v,\Dx\varphi)+(v,\Dx\varphi)-(\phi'(\Nx
u),\Nx\varphi)-(f(u),\varphi)+(g,\varphi).
$$
Using now that the $L^1$-norms of $f(u)$ and $\varphi'(\Nx u)$ are
controlled by the energy norm and the fact that $H^2\subset
L^\infty$, we deduce estimate \eqref{1.dt2}.
\par
 This estimate, together with the
interpolation inequality
\begin{equation}\label{1.interp}
\|\Dt v(t)\|_{H^{-1}}\le C\|\Dt v(t)\|_{H^{-3}}^{1/3}\|\Dt
v(t)\|_{L^2}^{2/3},
\end{equation}
allows us to control the right-hand side of \eqref{1.dtdif}
\begin{equation}\label{1.dtdiff}
\Dt E_{-1}(v(t))+\beta E_{-1}(v(t))+\frac12\beta(\|\Dt v\|^2_{L^2}+
\|\Nx v\|^2_{L^2})\le Q(\|\xi_u(t)\|_{\Cal E})
\end{equation}
for some new monotone function $Q$. Applying the Gronwall
inequality to this relation and using \eqref{1.energy} to control the
right-hand side, we derive the required estimate \eqref{1.dtest}
and finish the proof of the proposition.
\end{proof}
Finally, the next proposition shows that  the $\Dt u$-component of
the energy solution becomes more regular for $t>0$ (in a complete
agreement with the semi-linear case, see \cite{PZ}).

\begin{proposition}\label{Prop1.dtsmooth} Let the assumptions of
Theorem \ref{Th1.uniq} holds and let $u(t)$ be a weak energy solution of equation \eqref{1.main}. Denote
$v(t):=\Dt u(t)$. Then, $\xi_v(t)\in\Cal E_{-1}$ for any $t>0$ and the following
estimate holds:
\begin{multline}\label{1.dtsmooth}
\|\Dt u(t)\|^2_{H^1}+\|\Dt^2 u(t))\|_{H^{-1}}^2+\int_t^{t+1}\|\Dt^2
u(s)\|^2_{L^2}\,ds\le\\\le \frac{{1+t^N}}{t^N}Q(\|\xi_u(0)\|_{\Cal
E})e^{-\beta t}+Q(\|g\|_{L^2}),
\end{multline}
where $N>0$, $\beta>0$ are some constants and $Q$ is a monotone
function which are independent of $t$ and $u$.
\end{proposition}
\begin{proof} Indeed, due to estimate \eqref{1.dtest}, we only
need to verify \eqref{1.dtsmooth} for $t\in(0,1]$. Let us multiply
estimate \eqref{1.dtdiff} by $t^3$. Then, after the evident
transformations, we get
\begin{multline}\label{1.dtlast}
\Dt(t^3\|\xi_v(t)\|^2_{\Cal E_{-1}})+\beta
(t^3\|v(t)\|_{H^1}^2+t^3\|\Dt v(t)\|^2_{L^2})\le\\\le
Q(\|\xi_u(t)\|_{\Cal E})(1+\|\Dt\Nx u(t)\|^2_{L^2})+Ct^2\|\Dt
v\|_{H^{-1}}^2, \ t\in[0,1],
\end{multline}
for some new monotone function $Q$ and a constant $C$. We see that
all terms in the right-hand side except the last one are under the
control due to the energy estimate \eqref{1.energy}. In order to
estimate this term, we use the interpolation inequality
\eqref{1.interp} and the Young inequality as follows:
$$
t^2\|\Dt v\|^2_{H^{-1}}\le C t^2\|\Dt v\|^{4/3}_{L^2}\|\Dt
v\|_{H^{-3}}^{2/3}\le \eb t^3\|\Dt v\|_{L^2}^2+C^3\eb^{-2}\|\Dt^2
u\|_{H^{-3}}^2,
$$
where $\eb>0$ is small enough. Inserting this estimate into the
right-hand side of \eqref{1.dtlast}, integrating it over $t$ and
using \eqref{1.energy} and \eqref{1.dt2}, we arrive at estimate
\eqref{1.dtsmooth} (recall that $t\in(0,1]$) and finish the proof
of the proposition.
\end{proof}

\section{Weak global attractor and
finite-dimensionality}\label{s.atr}
In this section, we will construct the global attractor for the
semigroup $S(t)$ generated by weak energy solutions of
\eqref{1.main} and prove that it has finite fractal dimension.
Unfortunately, we are not able to construct the global attractor
in a strong topology of $\Cal E$, but only in the weaker topology
of $\Cal E_{-1}$.
\par
Namely, a set $\Cal A\subset\Cal E$ is a weak global attractor of the
solution semigroup $S(t):\Cal E\to\Cal E$ generated by equation
\eqref{1.main} (the so-called $(\Cal E,\Cal E_{-1})$-attractor in
the terminology of Babin and Vishik, see \cite{BV}) if
\par
1) The set $\Cal A$ is bounded in $\Cal E$ and compact in $\Cal
E_{-1}$.
\par
2) It is strictly invariant: $S(t)\Cal A=\Cal A$, $t\ge0$.
\par
3) The set $\Cal A$ attracts the images of bounded in $\Cal E$
sets in the topology of $\Cal E_{-1}$, i.e., for every bounded in
$\Cal E$ set $B$ and every neighborhood $\Cal O(\Cal A)$ of $\Cal
A$ in $\Cal E_{-1}$, there is a time $T=T(B,\Cal O)$ such that
$$
S(t)B\subset \Cal O(\Cal A),\ \ t\ge T.
$$
The following theorem which establishes the existence of such an
attractor for semigroup \eqref{1.sem} is the main result of this
section.

\begin{theorem}\label{Th2.fd} Let the assumptions of Theorem
\ref{Th1.dis} be satisfied. Then, the solution semigroup $S(t)$ generated
by equation \eqref{1.main} in the phase space $\Cal E$ possesses
a weak global attractor $\Cal A$ in the above sense. The attractor $\Cal A$ consists of all complete bounded trajectories of
equation \eqref{1.main}:
\begin{equation}\label{2.str}
\Cal A=\Cal K\big|_{t=0},
\end{equation}
where $\Cal K\subset L^\infty(\R,\Cal E)$ is a set of all
solutions $\xi_u(t)$ of equation \eqref{1.main} which are defined
for all $t\in\R$ and bounded: $\|\xi_u(t)\|_{\Cal E}\le C_u$,
$t\in\R$.
\par
Moreover, the attractor $\Cal A$ has finite fractal dimension in
$\Cal E_{-1}$:
\begin{equation}\label{2.fd}
\dim_F(\Cal A,\Cal E_{-1})\le C<\infty.
\end{equation}
\end{theorem}
\begin{proof} We first note that, due to the
dissipative estimate \eqref{1.energy}, the ball
$$
B_R:=\{\xi\in\Cal E,\ \ \|\xi\|_{\Cal E}\le R\}
$$
is an absorbing ball for the solution semigroup $\{S(t),\,t\ge0\}$
associated with equation \eqref{1.main} if $R$ is large enough.
Moreover, the same estimate guarantees also that the set
$$
\Cal B_R:=\big[\cup_{t\ge0}B_R\big]_{\Cal E_{-1}},
$$
where $[\cdot]_V$ denotes the closure in the space $V$,
will be bounded and closed in $\Cal E$ absorbing set satisfying
the additional semi-invariance property
\begin{equation}\label{2.inv}
S(t)\Cal B_R\subset\Cal B_R.
\end{equation}
Thus, it is sufficient to construct the global attractor $\Cal A$
for the semigroup $S(t)$ acting on $\Cal B_R$ only.
In order to do that, we need to refine  estimate
\eqref{1.dlip}.
\begin{lemma}\label{Lem2.ltr}
 Let $\xi_{u_1}(t)$ and $\xi_{u_2}(t)$ be
two trajectories of the semigroup $S(t)$ starting from the
absorbing set $\Cal B_u$ and let $v(t):=u_1(t)-u_2(t)$.
Then, the following estimate holds:
\begin{multline}\label{2.Ltr}
\|\xi_v(t)\|^2_{\Cal E_{-1}}+\int_s^t \|\Dt^2
v(\tau)\|^2_{H^{-3}}\,d\tau\le\\\le
Ce^{-\kappa(t-s)}\|\xi_v(s)\|^2_{\Cal
E_{-1}}+C\int_s^t(\|v(\tau)\|^2_{L^2}+\|\Dt
v(\tau)\|^2_{H^{-2}})\,d\tau,
\end{multline}
where the positive constants $C$ and $\kappa$ are independent
of $u_1$ and $u_2$ and $t\ge s\ge0$.
\end{lemma}
\begin{proof} Indeed, function $v(t)$
 solves equation \eqref{1.difer} and, arguing exactly
as in the proof of Theorem \ref{Th1.uniq}, we verify that estimate
\eqref{1.dlip} holds with the constant $C$ independent of the
choice of $\xi_{u_1}(0),\xi_{u_2}(0)\in\Cal E$. Using now estimate
\eqref{1.2sest} together with the interpolation inequality, we infer
from \eqref{1.dlip}  and \eqref{1.est} that
\begin{multline}\label{2.dlip}
\frac d{dt}E_{-1}(v(t))+\kappa E_{-1}(v(t))+
\kappa(\phi'(\Nx u_1(t))-\phi'(\Nx u_2(t)),\Nx v(t))+\\+\kappa(|f(u_1(t))-f(u_2(t))|,|v(t)|)\le
C(\|v(t)\|^2_{L^2}+\|\Dt v(t)\|^2_{H^{-2}}),
\end{multline}
where $\kappa$ and $C$ are fixed positive constants independent of
$v$. Applying the Gronwall inequality and using \eqref{1.e-1}
again, we  get
\begin{multline}\label{2.ltr}
\|\xi_v(t)\|^2_{\Cal E_{-1}}+\int_s^t(\phi'(\Nx u_1(\tau))-\phi'(\Nx u_2(\tau)),\Nx
v(\tau))+\\+
\kappa(|f(u_1(\tau))-f(u_2(\tau))|,|v(\tau)|)\,d\tau\le\\\le
C_1e^{-\kappa(t-s)}\|\xi_v(s)\|^2_{\Cal
E_{-1}}+C_1\int_s^t(\|v(\tau)\|^2_{L^2}+\|\Dt
v(\tau)\|^2_{H^{-2}})\,d\tau,
\end{multline}
where $\kappa$ and $C_1$ are independent of $v$ and $t\ge s\ge0$.
\par
Furthermore, arguing exactly as in \eqref{1.fest} and
\eqref{1.phiest}, we see that
\begin{multline}\label{2.der1}
\|\phi'(\Nx u_1)-\phi'(\Nx
u_2)\|_{L^1}^2+\|f(u_1)-f(u_2)\|_{L^1}^2\le\\\le C (\phi(u_1(t))-\phi(u_2(t)),\Nx v(t))
+C(|f(u_1(t))-f(u_2(t))|,|v(t)).
\end{multline}
Using equation \eqref{1.difer} in order to express the value of
$\Dt^2 v$ together with \eqref{2.ltr}, \eqref{2.der1} and the
embedding $L^1\subset H^{-2}$, we deduce the desired estimate
\eqref{2.Ltr} and finish the proof of the lemma.
 \end{proof}

The proved estimate \eqref{2.Ltr} is the key technical tool for the so-called
method of $l$-trajectories, see \cite{MP}. Indeed, let $L$ be a
sufficiently large number which will be fixed below and let
\begin{equation}\label{2.spaces}
\begin{cases}
\Cal H_1:=L^2([0,L],H^1_0(\Omega))\cap H^1([0,L],H^{-1}(\Omega))\cap H^2([0,L],H^{-3}),\\
\Cal H:=L^2([0,L],L^2(\Omega))\cap H^1([0,L],H^{-2}).
\end{cases}
\end{equation}
Then, evidently, the space $\Cal H_1$ is compactly embedded to
$\Cal H$. Introduce now the lifting operator
$$
\Bbb T_L:\Cal B_R\to\Cal H_1,\ \ \Bbb T_L\xi:=u(\cdot),
$$
where $u$ is an $L$-piece of the trajectory starting from
$\xi\in\Cal B_R$. Let also $\Bbb B^{tr}:=\Bbb T_L\Cal B_R$. Then,
the map $\Bbb T_L:\Cal B_R\to\Bbb B^{tr}$ is one-to-one and we can
consider the lift of the solution semigroup $\{S(t),\,t\ge0\}$ to
the $L$-trajectory phase space:
\begin{equation}\label{2.lift}
\Bbb S(t):\Bbb B^{tr}\to\Bbb B^{tr},\ \ \ \Bbb S(t):=\Bbb T_L\circ
S(t)\circ\Bbb T_L^{-1}
\end{equation}
The idea is now to verify the existence of the global attractor
$\Bbb A$ of the $L$-trajectory semigroup $\Bbb S(t)$ (more
precisely, of the discrete semigroup generated by the map $\Bbb
S(L)$) in the space $\Bbb B^{tr}$ instead of studying
 the initial semigroup $S(t)$. Indeed, although the map $\Bbb
 T_L^{-1}:\Bbb B^{tr}\to\Cal B_h$ is not continuous, it is not
 difficult to check that the map $T_L^{-1}\circ\Bbb S(L)$ is even Lipschitz
continuous. To this end, we need to write estimate \eqref{1.lip}
in the form
$$
\|\xi_v(L)\|_{\Cal E_{-1}}^2\le
Ce^{K(L-s)}\|\xi_v(s)\|^2_{\Cal E_{-1}}
$$
and integrate it over $s\in[0,L]$. By this reason, it is really
sufficient to prove the existence of a finite-dimensional
attractor $\Bbb A$ for the discrete semigroup generated by the map
$\Bbb S(L)$ only. The desired finite-dimensional global attractor
 $\Cal A$ of the initial semigroup can be
found then via
$$
\Cal A=\Bbb T_{L}^{-1}\circ\Bbb S(L)\Bbb A
$$
and the description \eqref{2.str} is a standard corollary of  the
attractor's
existence, see \cite{BV}.
\par
Finally, for proving the attractor existence for the map $\Bbb S(L)$,
we note that the basic estimate \eqref{2.Ltr} implies the
following version of the squeezing property for the map $\Bbb
S(L)$:
\begin{equation}\label{2.sqz}
\|\Bbb S(L)\xi_1-\Bbb S(L)\xi_2\|_{\Cal H_1}^2\le
\frac{C_1}L\|\xi_1-\xi_2\|_{\Cal H_1}^2+C_L(\|\xi_1-\xi_2\|_{\Cal
H}^2+\|\Bbb S(L)\xi_1-\Bbb S(L)\xi_2\|_{\Cal H}^2),
\end{equation}
where $\xi_1,\xi_2\in\Bbb B^{tr}$ and the constant $C_1$
 is independent of $L$. Indeed, integrating estimate \eqref{2.Ltr}
 over $s\in[0,L]$, and fixing $t\in[L,2L]$, we get
 $$
 L\int_L^{2L}\|\Dt^2 v(\tau)\|^2_{H^{-3}}d\,\tau\le  Ce^{-\kappa
 L}\|v(\cdot)\|^2_{\Cal
 H_1}+CL\int_0^{2L}(\|v(\tau)\|^2_{L^2}+\|\Dt
 v(\tau)\|^2_{H^{-2}})\,d\tau
 $$
 and
 $$
 L\|\xi_v(t)\|^2_{\Cal E^{-1}}\le Ce^{-\kappa(t-L)}
 \|v(\cdot)\|^2_{\Cal H_1}+CL\int_0^{2L}(\|v(\tau)\|^2_{L^2}+
 \|\Dt v(\tau)\|^2_{H^{-2}})\,d\tau.
 $$
 Integrating the last estimate over $t\in[L,2L]$ and using the
 previous one, we deduce \eqref{2.sqz}.
 \par
 Fix now $L$ in such way that $\alpha:=\frac {C_1}L<1$. Then, the
 smoothing property \eqref{2.sqz} together with the obvious
 Lipschitz continuity of the map $\Bbb S(L)$:
 $$
 \|\Bbb S(L)\xi_1-\Bbb S(L)\xi_2\|_{\Cal H_1}\le
 Ce^{KL}\|\xi_1-\xi_2\|_{\Cal H_1},\ \xi_i\in\Bbb B^{tr}
 $$
(which is also an immediate corollary of \eqref{2.Ltr}), imply
in a standard way the existence of the global attractor for
the map $\Bbb S(L)$ and it's finite-dimensionality in the space
$\Cal H_1$, see \cite{MP} (and also \cite{Z2,EFZ}). Theorem
\ref{Th2.fd} is proved.
\end{proof}

\begin{remark}\label{Rem2.noexp} Note that, usually, the above
described method of $L$-trajectories gives not only the existence
of a global attractor, but also of the {\it exponential} attractor
for the solution semigroup $S(t)$. However, in our case $\Cal E_{-1}$ is not compactly embedded in $\Cal E$ and an energy
solution is not H\"older continuous in time. This obstacle does
not allow to pass from the exponential attractor for the discrete
map $S(L)$ (which is factually constructed by the $L$-trajectory
technique) to the desired exponential attractor of the continuous semigroup
$S(t)$. We will overcome this obstacle below using the additional
smoothness of the global attractor.
\end{remark}

\begin{remark}\label{Rem2.batr} The space $\Cal
E_{-1}:=H^1_0(\Omega)\times H^{-1}(\Omega)$ in Theorem
\ref{Th2.fd} can be replaced by the better one $\tilde{\Cal
E}:=H^1_0(\Omega)\times L^2(\Omega)$. Indeed, due to Proposition
\ref{Prop1.dtsmooth}, we know that $\Dt u(t)\in H^1_0(\Omega)$ for
$t\ge0$. Using this fact and the proper interpolation inequality,
we obtain the following H\"older continuity:
$$
\|S(1)\xi_1-S(1)\xi_2\|_{\tilde{\Cal E}}\le
 C\|\xi_1-\xi_2\|_{\Cal E_{-1}}^{1/2},\ \xi_1,\xi_2\in\Cal B_R
$$
which immediately gives that result. However, we do not know how
to prove that the attractor $\Cal A$ attracts bounded sets in a
strong topology of the phase space $\Cal E$. For the particular
semilinear case ($\phi''(\Nx u)=const$ and $f(u)$ with an
 arbitrary growth rate) this result is established
in Section \ref{s4}.
\end{remark}

\section{Strong solutions: a priori estimates, existence and
dissipativity}\label{s2}
This section is devoted to strong solutions of equation
\eqref{1.main}. We start with the formal derivation of a
dissipative estimate in the space $\Cal E_1$ which will be
justified below. To this end, it is convenient to relax slightly
our assumptions on the non-linearity $\phi$  (we are
planning to prove the existence of a strong solution by approximating the growing non-linearities by the
non-growing ones and, by this reason, we need to allow the
non-growing non-linearities). Namely, we replace our assumptions
on $\phi$ by
\begin{equation}\label{3.phi}
\begin{cases}
1)\ \ \kappa_1|\phi''(v)|\cdot|w|^2\le \phi''(v)w.w\le
 \kappa_2(1+|\phi''(v))|)\cdot|w|^2,\ \ w\in\R^3,\\
2)\ \ |\phi''(v)|\cdot|v|^2\le C(1+\phi(v)),\\
3)\ \ \phi(v)\le C(1+\phi'(v). v),\\
4)\ \ |\phi'(v)|\le C(1+\phi(v))^{1/2}|\phi''(v)|^{1/2}.\\
\end{cases}
\end{equation}
Since we do not have now any growth restrictions on $\phi$, we
need to modify the energy norm:
$$
\|\xi_u\|_{\Cal E_\phi}^2:=\|\Dt u\|^2_{L^2}+\|\Nx
u\|^2_{L^2}+\|u\|^{q+2}_{L^{q+2}}+\|\phi(\Nx u)\|_{L^1}.
$$
The next lemma gives the formal dissipative estimate for the
$H^2$-norm of a strong solution $u$.

\begin{lemma}\label{Lem3.h2} Let the above asumptions hold and let
$u$ be a sufficiently regular solution of problem \eqref{1.main}.
Then, the following estimate holds:
\begin{multline}\label{3.h2}
\|u(t)\|_{W^{2,2}}^2+\int_t^{t+1}(|\phi''(\Nx u)|,|D^2_x u(s)|^2 )\,ds\le\\
\le Q(\|\xi_u(0)\|_{\Cal E_\phi}+\|u(0)\|_{W^{2,2}})e^{-\beta
t}+Q(\|g\|_{L^2}),
\end{multline}
where $D^2_xu$ means a collection of all second derivatives
$\partial_{x_i}\partial_{x_j}u$ of the function $u$) and the
monotone function $Q$ and the exponent $\beta>0$ depend only on
the constants $\kappa_i$ and from \eqref{3.phi} (and independent of
the concrete choice of $\phi$ satisfying these conditions).
\end{lemma}
\begin{proof}
The key idea here is to multiply (analogously to the semi-linear
 case, see e.g., \cite{PZ,K88}) the
equation \eqref{1.main} by $\Dx u$ and the main difficulty is the
quasi-linear term. In the case of {\it periodic} boundary
conditions not any additional problems arise, since
\begin{multline}\label{3.conv}
\sum_i(\Nx\phi'(\Nx u),\partial_{x_i}^2 u)=\sum_{i,j}(\partial_{x_i} \phi'_j(\Nx
u),\partial_{x_i}\partial_{x_j} u)=\\=\sum_i(\phi''(\Nx u)\partial_{x_i}\Nx
u,\partial_{x_i}\Nx u)\ge \kappa(|\phi''(\Nx u)|,|D^2_xu|^2),
\end{multline}
where we have used the first inequality of \eqref{3.phi}. However, this does not work directly for the
case of Dirichlet boundary conditions due to the appearance of uncontrollable
boundary terms under the integration by parts.
\par
In order to overcome this difficulty, we will use some kind of
localization technique. First, we deduce the $H^2$-estimate {\it
inside} of the domain $\Omega$. To this end, we multiply equation
\eqref{1.main} by $\Nx(\theta(x)\Nx u)$ where $\theta(x)$ is a cut-off
function which equals zero near the boundary and one in the
$\delta$-interior of the domain $\Omega$ and satisfies the
inequality
$$
|\theta''(x)|+|\theta'(x)|\le C[\theta(x)]^{1/2}.
$$
 Then, analogously to
\eqref{3.conv}, we will have
\begin{multline}\label{1.int1}
(\Nx\phi'(\Nx u),\Nx(\theta\Nx u))=\sum_{i,j}(\partial_{x_i}\phi'_j(\Nx u)),\partial_{x_i}(\theta\partial_{x_j} u))\ge
\\\ge\kappa(\theta|\phi''(\Nx u)|,|D^2_x u|^2)-C(|\phi'(\Nx u)|,(|\theta''|+|\theta'|)|D^2_x u|)\ge\\\ge 1/2\kappa(\theta|\phi''(\Nx u)|,|D^2
u|^2)-C_1(1+\|\phi(\Nx u)\|_{L^{1}}).
\end{multline}
Here we
 have used assumption \eqref{3.phi}(4) and the following
estimate
\begin{multline}\label{3.stan}
(|\phi'(\Nx u)|,(|\theta''|+|\theta'|) |w|)=\\=
(|\phi''(\Nx v)|^{-1/2}|\phi'(\Nx u)|,(|\theta''|+|\theta'|)|\phi''(\Nx u)|^{1/2}\cdot|w|)
\le\\\le (|\phi''(\Nx
u)|^{-1}\cdot|\phi'(\Nx u)|^2,1)^{1/2}((|\theta''|^2+|\theta'|^2)|\phi''(\Nx u)|,|w|^2)^{1/2}\le\\\le
1/2\kappa(\theta|\phi''(\Nx u)|,|w|^2)+C(1+\|\phi(\Nx u)\|_{L^1}).
\end{multline}
The other terms are much simpler to estimate. For instance, due to
the quasi-monotonicity assumption $f'(u)\ge-K$, the other nonlinear
term  factually disappears
$$
-(f(u),\Nx(\theta\Nx u))=(\theta f'(u)\Nx u,\Nx u)\ge -K\|\Nx
u\|^2_{L^2}
$$
and the linear terms can be estimated as follows:
\begin{multline}
(\Dt\Dx u,\Nx(\theta\Nx u))=\frac12\Dt(\theta,|\Dx u|^2)-(\Dt\Nx
u,\Nx(\theta'\Nx u)\ge\\\ge\frac12\Dt(\theta,|\Dx u|^2)-C_\eb\|\Dt\Nx
u\|^2-\eb(\theta,|\Dx u|^2)-C\|\Nx u\|^2_{L^2}
\end{multline}
and
\begin{multline}
-(\Dt^2 u,\Nx(\theta\Nx u))=(\Dt^2\Nx u,\theta\Nx u)=\Dt (\Dt \Nx
u,\theta\Nx u)-\\-(\theta\Dt\Nx u,\Dt\Nx u)\ge-\Dt (\Dt
u,\Nx(\theta\Nx u))-C\|\Dt\Nx u\|^2.
\end{multline}
Thus, combining the above estimates, we arrive at
\begin{multline}\label{3.gr}
\frac d{dt}[\frac\gamma2(\theta,|\Dx u|^2)-(\Dt u,\Nx(\theta(\Nx u))]+
[(\theta,|\Dx u|^2)-(\Dt u,\Nx(\theta(\Nx u))]+\\+\frac12\kappa(\theta|\phi''(\Nx u)|,|D^2_x u|^2 )\le
  C(1+\|\phi(\Nx u)\|_{L^1}+\|\Dt\Nx u\|^2_{L^2}+\|\Nx u\|^2_{L^2}).
\end{multline}
Applying the Gronwall inequality to this relation and using the
following
analog of the dissipative estimate \eqref{1.energy}
\begin{multline}\label{3.phien}
\|\Dt u(t)\|^2_{L^2}+\|u(t)\|^2_{H^2}+\|\phi(\Nx u(t))\|_{L^1}+\|u(t)\|^{q+2}_{L^{q+2}}+\int_t^{t+1}\|\Dt\Nx
u(s)\|^2_{L^2}\,ds\le\\\le
Q(\|\xi_u(0)\|_{\Cal E_\phi})e^{-\gamma t}+C(1+\|g\|^2_{L^2})
\end{multline}
 in order to
control  the right-hand side of \eqref{3.gr}, we deduce the required $H^2$-estimate inside of the
domain $\Omega$:
\begin{multline}\label{1.h2int}
\|\theta u(t)\|_{W^{2,2}}^2+\int_t^{t+1}(\theta|\phi''(\Nx u)|,|D^2_x u(s)|^2)\,ds\le
\\ \le Q(\|\xi_u(0)\|_{\Cal E_\phi}+\|u(0)\|_{W^{2,2}})e^{-\beta
t}+Q(\|g\|_{L^2}).
\end{multline}
At the next step, we consider the neighborhood of the boundary
$\partial\Omega$. To this end, we extend the tangent ($\tau_1=(\tau^1_1,\tau_2^2,\tau_1^3)$ and
$\tau_2=(\tau_2^1,\tau_2^2,\tau_2^3)$) and normal ($n=(n^1,n^2,n^3)$) vector fields from the boundary
$\partial\Omega$ inside of the domain $\Omega$ in such way that
the basis $(\tau_1(x),\tau_2(x),n(x))$ will be orthonormal at least in a
small neighborhood of the boundary
(being pedantic, in a small neighborhood of an
 arbitrarily fixed point $x_0\in\partial\Omega$; outside of this neighborhood, the above vector fields should be cutted off using,
 e.g., the scalar cut-off multiplier $\theta(x)$). Let $\partial_{\tau_i}:=\sum_{j=1}^3\tau_i^j(x)\partial_{x_j}$ and
$\partial_n:=\sum_{j=1}^3n^j(x)\partial_{x_j}$ be the associated differential
operators. Let also $\partial^*_{\tau_i}$, $i=1,2$, be the
formally adjoint (in $L^2(\Omega)$) differential operators, i.e.,
$$
(\partial_{\tau_i}^*v)(x):=-\sum_{j=1}^3\partial_{x_j}(\tau_i^j(x)v(x)).
$$
\par
We now multiply the equation by the term
$$
-\partial_{\tau_i}^*\partial_{\tau_i}
u=\sum_{j,k}\partial_{x_k}(\tau_i^k(x)\tau_i^j(x)\partial_{x_j}u).
$$
Then, since
$D_\tau u\big|_{\partial\Omega}=0$ (here and below,  $D_\tau$ stands for the one of tangent
derivatives $\partial_{\tau_1}$ or $\partial_{\tau_2}$), we do not have any
boundary terms under the integrating by parts in the quasi-linear
term. Moreover, we also have
$$
|\partial_{\tau_i}\Nx F-\Nx\partial_{\tau_i} F|\le C|\Nx F|.
$$
By this reason, analogously to \eqref{1.int1} and \eqref{3.stan}, we can estimate the quasi-linear term by
\begin{multline}\label{1.tan}
-(\Nx\phi'(\Nx u),D_{\tau}^*D_{\tau} u)=-(D_{\tau}\Nx\phi'(\Nx u),D_{\tau}u)\ge
\\\ge-(\Nx D_{\tau}\phi'(\Nx
u),D_{\tau}u)-C(|\phi'(\Nx u)|,|\Nx D_{\tau} u|)=\\
=(D_{\tau}\phi'(\Nx u),\Nx D_{\tau} u)-
C(|\phi'(\Nx u)|,|\Nx D_{\tau} u|)\ge\\\ge
\kappa(|\phi''(\Nx u)|,|\Nx D_\tau u|^2)-
\\-C(|\phi'(\Nx u)|,|\Nx D_\tau u|)\ge
 1/2\kappa(|\phi''(\Nx u)|,|\Nx D_\tau u|^2)-C_1(1+\|\phi(\Nx u)\|_{L^1}).
\end{multline}
The linear terms are easier to estimate, in particular,
$$
-(\Dt\Dx u,D_\tau^*D_\tau u)\ge \frac12\frac d{dt}\|\Nx
D_\tau\|^2_{L^2}-C\|\Dt\Nx u\|_{L^2}\|\Nx D_\tau u\|_{L^2}.
$$
Then, arguing exactly as in the derivation of \eqref{1.h2int}, we obtain the control of second derivatives of $u$ in
tangential directions:
\begin{multline}\label{1.h2tan}
\|\Nx D_\tau u(t)\|_{L^2}^2+\int_t^{t+1}(|\phi''(\Nx u(s))|,|\Nx D_\tau u(s)|^2,)\,ds\le\\
\le Q(\|\xi_u(0)\|_{\Cal E_\phi}+\|u(0)\|_{W^{2,2}})e^{-\beta t}+Q(\|g\|_{L^2}).
\end{multline}
Thus, we only need to estimate the second normal derivatives. To
this end, we multiply the equation \eqref{1.main} by $\Dx u$ and,
instead of integrating by parts in the quasi-linear term, will now
use the fact that all of the second derivatives except of $\partial_n^2 u$
are already under the control (due to \eqref{1.h2tan} and
\eqref{1.h2int}). More precisely, let $n:=\tau_3$ and
$$
\partial_{x_i}=A_{ik}(x)\partial_{\tau_k},\ \ i,k=1,2,3.
$$
Then, since the vector fields $(\tau_1,\tau_2,n)$ are
orthonormal (at least near the boundary), we may conclude that
$$
|\Dx u-\partial^2_n u|\le C(|\Nx D_\tau|^2+|\Nx
u|^2+|u|^2+|\theta D^2_x u|^2),
$$
where  $\theta$
is the cut-off function defined above. Moreover, analogously
\begin{multline*}
|\Nx\phi'(\Nx u)-(\sum_{ij}\phi''_{ij}(\Nx u)A_{i3}A_{j3})\partial^2_n
u|\le C(|\phi''(\Nx u)|\cdot|\Nx D_{\tau} u|^2+\\+\theta|\phi''(\Nx
u)|\cdot|D^2_x u|^2+|\phi''(\Nx u)|\cdot|\Nx u|^2)
\end{multline*}
and, due to \eqref{3.phi}(1)
$$
\sum_{ij}\phi''_{ij}(\Nx u)A_{i3}A_{j3}\ge\kappa|\phi''(\Nx u)|
$$
(at least in the small neighborhood of the boundary).
Using that estimates, we finally infer
\begin{multline}\label{1.norm}
(\Nx\phi'(\Nx u),\Dx u)\ge \kappa(|\phi''(\Nx u)|,|\partial_n^2 u|^2)-
C(|\phi''(\Nx u)|,|\Nx D_{\tau} u|^2)-\\-
C(1+\|\phi(\Nx u)\|_{L^1})-C(\theta|\phi''(\Nx u)|,|D^2_{x} u|^2).
\end{multline}
Using now estimates \eqref{1.h2tan} and
\eqref{1.h2int} together with \eqref{3.phien} in order to control
the subordinated terms, we get
\begin{multline*}
\|\Dx u(t)\|^2_{L^2}+\int_{t}^{t+1}(|\phi''(\Nx
u(s))|,|\partial_n^2 u(s)|^2)\,ds\le\\\le Q(\|\Dx
u(0)\|_{L^2}+\|\xi_u(0)\|_{\Cal E_\phi})e^{-\beta t}
+C(1+\|g\|^2_{L^2})
\end{multline*}
which together with \eqref{1.h2tan} and
\eqref{1.h2int} give the desired estimate \eqref{3.h2} and finish
the proof of the lemma.
\end{proof}

We are now ready to prove the existence of $H^2$-solutions. In
contrast to the previous section, the Galerkin method seems not
applicable here since the multiplication of the equation by
$\partial_{\tau}^*\partial_{\tau} u$ will be problematic on the
level of Galerkin approximations. By this reason, we will proceed
in an alternative way approximating the non-linearity $\phi$ by
the sequence  $\phi_n$ such that $\phi''_n$ are globally bounded.
For this reason, we needed the modified assumptions \eqref{3.phi}.
We formulate the main result of this section in the following
theorem.

\begin{theorem}\label{Th3.strong} Let the nonlinearities $\phi$
and $f$ satisfy assumptions \eqref{3.int} (with arbitrary growth
exponents $p$ and $q$!). Then, for every $\xi_u(0)\in\Cal E$ such
that $u(0)\in H^2(\Omega)$, there exists at least one weak energy
solution $\xi_u(t)$ with the additional regularity $u(t)\in
H^2(\Omega)$ which satisfies the following estimate:
\begin{equation}\label{3.h2en1}
\|\xi_u(t)\|_{\Cal E}+\|u(t)\|_{H^2}\le Q(\|\xi_u(0)\|_{\Cal
E}+\|u(0)\|_{H^2})e^{-\beta t}+Q(\|g\|_{L^2}),
\end{equation}
for some positive constant $\beta$ and monotone function $Q$.
\end{theorem}
\begin{proof} The proof of solvability is based on
the following observation: if the nonlinearity $\phi(v)$ is such
that $\phi''(v)$ is globally bounded and the initial data is
smooth enough, say $\xi_u(0)\in\Cal E_1$, then, for any $H^2$-solution $u(t)$ all of the terms
in equation \eqref{1.main} belong to $L^2([0,T]\times
L^2(\Omega))$. Indeed, since $\phi''(v)\le C$, then $u\in H^2$
implies that $\Nx\phi'(\Nx u)\in L^2$ and $\Dt^2 u\in
L^2([0,T]\times\Omega)$ due to Proposition \ref{Prop1.dtreg}. Thus, all
terms except of $\Dt\Dx u$ in \eqref{1.main} belong indeed to
$L^2$ and, therefore, the term $\Dt\Dx u$ should also belong to
that space.
\par
By this reason,
we are able to multiply equation \eqref{1.main} by $\Dx u$ and
$\partial_\tau^*\partial_\tau u$ and, therefore, the formal
computations can be justified in a standard way if, in addition,
\begin{equation}\label{3.appinit}
|\phi''(v)|\le C,\ \ \ \xi_u(0)\in\Cal E_{1}.
\end{equation}
Thus, we will approximate the nonlinearity $\phi(v)$  by a sequence
$\{\phi_n(v)\}$ which is globally bounded in such way that
assumptions \eqref{3.phi} hold {\it uniformly} with respect to
$n\to\infty$. Although it is more or less clear that such a sequence
exists, for the convenience of the reader we will briefly describe
below one possible way to construct it.
\par
 At the first step, we note that such a sequence
obviously exists in the particular case
$$
\phi_r(v):=|v|^r,
$$
for any fixed $r\ge2$. For instance, we can take
$\phi_N(v):=\theta_N(\phi(v))$ where the cut-off function is such
that
\begin{equation*}
\theta_N(x)=\begin{cases}x, \ {\rm for}\ \  x\le N;\\
A_N+B_Nx^{2/r},\ \ {\rm for }\ \ x\ge N.
\end{cases}
\end{equation*}
where the coefficients $A_N$ and $B_N$ are such that $\theta_N\in
C^1$ near $x=N$ (being pedantic, the constructed function has a jump of the second derivative
at $x=N$, but it is not essential since all above estimates work for such type of non-linearities $\phi$).
\par
At the second step, we introduce a family of functions
$$
\phi_{\eb}(v):=\phi(v)+\eb\phi_r(v),
$$
where the exponent $r$ is {\it larger} than $p+1$ (recall that the
non-linearities $\phi$ and $f$ satisfy \eqref{1.int} and \eqref{3.int}) and observe
that assumptions \eqref{3.phi} hold for the functions
$\phi_{\eb}(v)$ {\it uniformly} with respect to $\eb\to0$.
\par
Finally, at the third step we construct the desired approximations
$$
\phi_n(v):=\theta_{N_n}(\phi(v)+\eb_n\phi_r(v)),
$$
where the sequence $\eb_n\to0$ (which guarantees the uniform
convergence of $\phi_n$ to $\phi$ at every compact set) and the
sequence $N_n$ tends to infinity fast enough in order to guarantee
that the cut-off  starts at the region where
$$
\eb_n\phi_r(v)\gg\phi(v).
$$
Therefore, in that region, the leading term of $\phi_n(v)$ is
$\theta_{N_n}(\eb_n\phi_r(v))$ for which \eqref{3.phi} are
satisfied {\it uniformly} with respect to $n$. By this reason,
these assumptions will be satisfied also for the perturbed
functions $\phi_n(v)$ and also uniformly with respect to
$n\to\infty$.
\par
Thus, we have constructed the family $\phi_n(v)$ with the
following properties:
\par
1) $\phi_n$ satisfy \eqref{3.phi} uniformly with respect to
$n\to\infty$;

2) $\phi_n(v)\to\phi(v)$ uniformly with respect to all $v\in\R^3$ such that
$|v|\le R$ ($R$ is arbitrary);

3) The following additional inequalities hold:
\begin{equation}\label{3.bad}
1.\  |\phi_n(v)|\le C(1+|v|^r),\ \ \ 2.\ |\phi'_n(v)|\le
C(1+\phi(v))^{r/(r+1)},
\end{equation}
where the constant $C$ is independent of $n$.
\par
We are now ready to verify the existence of the desired solution. To
this end, we first consider the case of regular initial data
$\xi_u(0)\in C^2(\Omega)\times C^2(\Omega)$ and construct a (unique)
approximate solution $u_n$ by solving the following system:
\begin{equation}\label{3.appro}
\Dt^2 u_n-\gamma\Dt\Dx u_n-\Dx u_n+f(u_n)=\Nx\phi_n'(\Nx u_n)+g,\
\xi_{u_n}(0)=\xi_u(0).
\end{equation}
Then, due to Lemma \ref{Lem3.h2}, we have the uniform (with
respect to $n$) estimate
\begin{equation}\label{3.h21}
\|u_n(t)\|_{W^{2,2}}^2
\le Q(\|\xi_u(0)\|_{\Cal E_{\phi_n}}+\|u(0)\|_{W^{2,2}})e^{-\beta t}+Q(\|g\|_{L^2}).
\end{equation}
In particular, since $\xi_u(0)$ is smooth, the first estimate in
\eqref{3.bad} guarantees that
$$
\|\xi_u(0)\|_{\Cal E_{\phi_n}}\to\|\xi_u(0)\|_{\Cal E_\phi}.
$$
It is now not difficult to pass to the limit in equations
\eqref{3.appro}. Indeed, due to the uniform estimate \eqref{3.h2},
we may assume without loss of generality that $\xi_{u_n}$
converges weakly-star to some function $\xi_u$ in the space
$$
L^\infty([0,T],[H^2(\Omega)\cap H^1_0(\Omega)]\times L^2(\Omega))
$$
and we need to prove that $u$ solves \eqref{1.main} by passing
 to the limit in \eqref{3.appro}. As usual, the passage
to the limit in the linear terms are immediate and we only need to
take care on the non-linear ones. Note that the above weak
convergence implies the strong convergence
\begin{equation}\label{3.stcon}
u_n\to u\ \ {\rm in \  the \ space }\ \
C([0,T],H^{2-\eb}(\Omega))\subset C([0,T]\times\Omega)
\end{equation}
and, therefore, $f(u_n)\to f(u)$. So, we only need to check
that
$$
\Nx\phi'_n(\Nx u_n)\to\Nx\phi'(\Nx u)
$$
in the sense of distributions. To this end, we note that, due to
\eqref{3.bad},
$$
\|\phi'(\Nx u)\|_{L^{(r+1)/r}}\le C.
$$
Moreover, from \eqref{3.stcon} we conclude that $\Nx u_n\to\Nx u$
almost everywhere in $[0,T]\times\Omega$ and, therefore,
$\phi'_n(\Nx u_n)\to\phi'(\Nx u)$ almost everywhere. Thus, we have
proved that $\phi'_n(\Nx u_n)\to\phi'(\Nx u)$ in $L^{(r+1)/r}$ and
that $u$ solves indeed equation \eqref{1.main} in the sense of
distributions.
\par
Let us check that the solution $u$ belongs to
$L^\infty([0,T],W^{1,p+1}(\Omega))$ and satisfies the dissipative estimate
\eqref{3.h2en1}. To this end, it is sufficient to note that, due
to the convergence $\phi'_n(\Nx u_n)\to\phi'(\Nx u)$ almost everywhere, the Fatou lemma gives
\begin{equation}\label{3.convp}
\|\phi(\Nx u)\|_{L^1}\le
\liminf_{n\to\infty}\|\phi'_n(\Nx)\|_{L^1}
\end{equation}
which allows to verify \eqref{3.h2en1} by passing to the limit $n\to\infty$ in estimates
\eqref{3.h21} (exactly, in order to pass to the limit in the
right-hand side, we need the additional assumption  on the
initial data to be  smooth).
\par
Thus, we have constructed the desired $H^2$-solution of problem
\eqref{1.main} under the additional assumption that
$$
(u(0),\Dt u(0))\in C^2(\Omega)\times C^2(\Omega).
$$
Finally, this assumption can be easily removed by approximating
the arbitrary $H^2$-initial data by the smooth one and passing to
the limit once more. Theorem \ref{Th3.strong} is proved.
\end{proof}

\begin{remark} Passing to the limit  $n\to\infty$ in the
dissipative estimate \eqref{3.h2} for the approximations $u_n$, we
can prove that the limit solution $u(t)$  also satisfies
\begin{equation}\label{3.ad1}
\int_t^{t+1}(|\phi''(\Nx u(s))|,|D^2_x u(s)|^2)\,ds\le
Q(\|\xi_u(0)\|_{\Cal E}+\|\Dx u(0)\|_{L^2})e^{-\gamma
t}+Q(\|g\|_{L^2})
\end{equation}
which, together with our growth assumptions on $\phi$ and Sobolev
embedding theorem, gives the following control
\begin{equation}\label{3.ad2}
\|\Nx u\|_{L^{p+1}([T,T+1],L^{3(p+1)}(\Omega))}\le Q(\|\xi_u(0)\|_{\Cal E}+\|\Dx u(0)\|_{L^2})e^{-\gamma
t}+Q(\|g\|_{L^2})
\end{equation}
which is important for establishing the uniqueness of strong
solutions for the limit growth exponent $p=5$, see below.
\end{remark}

Note once more that the above proved theorem does not require any
growth restrictions on $\phi$. However, in that case, we can not
guarantee the uniqueness. Combining now Proposition
\ref{Prop1.dtreg} with Theorem \ref{Th3.strong}, we obtain the
following result on the well-posedness and dissipativity of strong
solutions.
\begin{corollary}\label{Cor3.strong} Let the assumptions of Theorem \ref{Th1.uniq}
hold (in particular $p<5$ now). Then, for every $\xi_u(0)\in\Cal
E_1$ (see \eqref{1.e1energy1} for the definition), there exists a
unique strong solution $u(t)$ of problem \eqref{1.main} and the
following estimate holds:
\begin{multline}\label{3.strongdis}
\|\Dt u(t)\|^2_{H^1}+\|\Dt^2 u(t)\|^2_{H^{-1}}+\\+\|\Nx \phi'(\Nx
u(t))\|_{H^{-1}}^2+
\|u(t)\|^2_{H^2}\le Q(\|\xi_u(0)\|_{\Cal E_1})e^{-\beta
t}+Q(\|g\|_{L^2})
\end{multline}
for some positive constant $\beta$ and monotone function $Q$.
\end{corollary}
Indeed, estimate \eqref{3.strongdis} is an immediate corollary of
\eqref{1.dtest} and \eqref{3.h2en1} (the assumption  $\Nx\phi'(\Nx u(0))\in H^{-1}$
guarantees that $\Dt^2 u(0)\in H^{-1}$ and Proposition \ref{Prop1.dtreg}
 is indeed applicable; vise versa, the control of the $H^{-1}$-norm of $\Dt^2 u$ obtained from Proposition \ref{Prop1.dtreg}
 together with the $H^2$-estimate allows to control the $H^{-1}$-norm of $\Nx \phi'(\Nx u(t))$).

\begin{remark}\label{Rem3.space} We see that the additional
non-linear condition $\Nx \phi'(\Nx u)\in H^{-1}$ is included to
the definition \eqref{1.e1energy1} of the space $\Cal E_1$ in
order to be able to control the $H^{-1}$-norm of the second time
derivative $\Dt^2 u(t)$ of the solution (which is crucial for our
method). In the case $p\le 3$, the assumption $u\in H^2$
automatically implies that $\phi'(\Nx u)\in L^2(\Omega)$ and,
therefore, $\Nx\phi'(\Nx u)\in H^{-1}$. So, this additional
condition is not necessary if $p\le3$ and the strong energy space
$\Cal E_1$ has a usual form
$$
\Cal E_1=[H^2\cap H^1_0]\times H^1_0.
$$
However, it is not so if the non-linearity $\phi$ grows faster
($p>3$) since in that case the $W^{1,2p}$-norm of the solution is
no more controlled by its $H^2$-norm and the assumption
$\Nx\phi'(\Nx u)\in H^{-1}$ gives the additional information about
the regularity of the solution which can not be obtained from the
linear terms. Indeed, assume that the quasi-linear elliptic
problem
\begin{equation}\label{3.qel}
\Nx\cdot\phi'(\Nx U)=h,\ \ U\big|_{\partial\Omega}=0,\ \ h\in
H^{-1}(\Omega)
\end{equation}
possesses the maximal regularity theorem in the form
\begin{equation}\label{3.elreg}
\|\phi'(\Nx U)\|_{L^2}\le Q(\|h\|_{H^{-1}})
\end{equation}
for every $h\in H^{-1}$ and some monotone function $Q$
(for the particular case $\phi(z):=|z|^{p+1}$, this regularity
result is known to be true, see \cite{AG}). In that case, due to our
growth assumptions on $\phi$, the additional condition $\Nx
\phi'(\Nx u)\in H^{-1}$ is simply equivalent to $u\in
W^{1,2p}(\Omega)$ and the strong energy space $\Cal E_1$ now reads
\begin{equation}\label{3.ssimple}
\Cal E_1=[H^2(\Omega)\cap W^{1,2p}_0(\Omega)]\times H^1_0(\Omega).
\end{equation}
Unfortunately, we do not know whether or not the regularity
estimate \eqref{3.elreg} hold for the general non-linearity $\phi$
satisfying \eqref{1.int} and, by this reason, have to write the
strong energy space $\Cal E_1$ in the form of \eqref{1.e1energy1}.
\end{remark}

To conclude the section, we briefly discuss how to handle the
limit case $p=5$. In that case, for obtaining the uniqueness, we
need to include the control $\Nx u\in L^6([0,T],L^{18}(\Omega))$
(see \eqref{3.ad2}) into the definition of a strong solution.

\begin{proposition}\label{Prop3.limit} Let the assumptions \eqref{1.int} and \eqref{3.int}
be satisfied with $p=5$. Then, the strong solution $u(t)$ of problem
\eqref{1.main} which satisfies the additional regularity
\begin{equation}\label{3.ad3}
\Nx u\in L^6([0,T],L^{18}(\Omega))
\end{equation}
is unique.
\end{proposition}
\begin{proof} Indeed, in the proof of Theorem \ref{Th1.uniq},
we have used the assumption $p<5$ in order to estimate the term
$$
([\phi'(\Nx u_1)-\phi'(\Nx u_2)],\Nx(-\Dx)^{-1}\Dt v)
$$
only. Instead, we now estimate this term based on \eqref{3.ad3}
and using the H\"older inequality with exponents $3$, $2$ and $6$
as follows:
\begin{multline}\label{3.new}
(\int_0^1\phi''(s\Nx u_1+(1-s)\Nx u_2)\,ds\Nx v,\Nx(-\Dx)^{-1}\Dt v)\le\\\le
 C(1+\|\Nx u_1\|_{L^{12}}+\|\Nx u_2\|_{L^{12}})^4\|\Nx v\|_{L^2}\|\Dt
 v\|_{L^2}\le\\\le
 \eb\|\Dt v\|^2_{L^2}+C_\eb(1+\|\Nx u_1\|_{L^{12}}+\|\Nx u_2\|_{L^{12}})^8\|\Nx
 v\|^2_{L^2}.
\end{multline}
Since, due to the interpolation inequality,
$$
L^\infty([0,T],H^1)\cap L^6([0,T],L^{18}(\Omega))\subset
L^8([0,T],L^{12}(\Omega)),
$$
assumption \eqref{3.ad3} together with the fact that the $H^2$
norm of $u(t)$ is bounded, gives the control
$$
\int_0^T\|\Nx u(s)\|^8_{L^{12}}\,ds\le C_T.
$$
Thus, analogously to \eqref{1.dlip},  we get
\begin{multline}\label{1.der}
\Dt [\gamma/2\|\Nx v\|_{L^2}^2+(\Dt v,v)+\frac2\gamma(\|\Dt
v\|_{H^{-1}}^2+\|v\|^2_{L^2})]+\beta(\|\Dt v\|^2_{L^2}+\|\Nx
v\|^2_{H^1})
\le\\\le C(1+\|\Nx u_1\|_{L^{12}}+\|\Nx u_2\|_{L^{12}})^8(\|\Nx v\|^2+\|\Dt v\|^2_{H^{-1}}).
\end{multline}
Applying the Gronwall inequality to this relation, we will have
\begin{equation}\label{1.derdiv}
\|\Dt v(t)\|_{H^{-1}}^2+\|v(t)\|_{H^1}^2\le
 C(t)(\|\Dt v(0)\|_{H^{-1}}^2+\|v(0)\|_{H^1}^2).
\end{equation}
which gives the desired uniqueness.
\end{proof}

\begin{corollary}\label{Cor3.dtreg} Let the assumptions of
Proposition \eqref{Prop3.limit} hold. Then, estimate
\eqref{1.dtest} is satisfied.
\end{corollary}
\begin{proof} Indeed, arguing  as in the proof of Proposition
\ref{Prop1.dtreg}, but using estimate \eqref{3.new}, we derive the
analog of \eqref{1.dtest}, but with the {\it growing} in time
coefficients, see \eqref{1.der} and \eqref{1.derdiv}.
\par
In order to overcome this difficulty, we need to combine this
estimate with the smoothing property of Proposition
\ref{Prop1.dtsmooth} which can be obtained by multiplying
inequality \eqref{1.der} (of course, with $v=\Dt u$) by $t^3$ and
arguing exactly as in \eqref{1.dtlast}. This finishes the proof of
the proposition.
\end{proof}

\begin{corollary}\label{Cor3.lastdis} Let the assumptions of
Proposition \ref{Prop3.limit} be satisfied and let
$\xi_u(0)\in\Cal E_1$. Then, the following dissipative estimate
holds for the unique strong solution $u(t)$ of problem
\eqref{1.main}
\begin{multline}\label{1.strongdis}
\|u(t)\|^2_{H^2}+\|\Dt u(t)\|^2_{H^1}+\|\Dt^2 u(t)\|_{H^{-1}}^2+\\+\int_t^{t+1}\|\Nx u(s)\|^6_{L^{18}}\,ds
\le Q(\|\xi_u(0)\|_{\Cal E_1})e^{-\beta t}+Q(\|g\|_{L^2}),
\end{multline}
for some positive constant $\beta$ and monotone function $Q$.
\end{corollary}
Indeed, the above estimate is a combination of \eqref{3.h2en1},
\eqref{3.ad2} and \eqref{1.dtest}.

\section{Regularity of the global attractor and the exponential
attractor}\label{s4}
In this section, we verify that the weak global attractor $\Cal A$
consists of strong solutions and construct an exponential
attractor for the semigroup $\{S(t)\,, t\ge0\}$ associated with
equation \eqref{1.main} in the energy space $\Cal E$. We start
with the following theorem which establishes the asymptotic
regularity of that semigroup.

\begin{theorem}\label{Th4.exp} Let the assumptions of Theorem
\ref{Th1.uniq} hold. Then, for every bounded set $B$ in $\Cal E$,
the following estimate is satisfied:
\begin{equation}\label{4.exp}
\dist_{\Cal E_{-1}}(S(t) B,\{\xi\in\Cal E_1,\ \|\xi\|_{\Cal E_1}\le K\})\le Q(\|B\|_{\Cal E})e^{-\beta t}
\end{equation}
if $K$ is large enough. Here
where $\beta$ is some positive number, $Q$ is a monotone function
both independent of $B$ and $t$ and $\dist_V(X,Y)$ is a
non-symmetric Hausdorff distance between sets in $V$.
\end{theorem}
\begin{proof} Due to the dissipative estimate \eqref{1.energy}, it
is sufficient to prove the theorem for the absorbing ball $B=B_R$.
\par
Let now $u(t)$ be any trajectory starting from the absorbing ball
$B_R$. Then, due to Proposition \ref{Prop1.dtsmooth}, we may assume
without loss of generality that
\begin{equation}\label{4.ap}
\int_t^{t+1}\|\Dt^2 u(s)\|^2\,ds+\|\Dt^2 u(t)\|^2_{H^{-1}}+\|\Dt
u(t)\|_{H^1}^2\le C_R,
\end{equation}
where the constant $C_R$ is independent of $t$. Thus, it is
sufficient to verify that
\begin{equation}\label{4.simexp}
\dist_{H^1}(u(t),\{u_0\in H^2\cap H^1_0,\ \|u_0\|_{H^2}\le K\})
\le C_Re^{-\beta t}.
\end{equation}
In order to do so we split the function $u$ as follows:
$u(t)=v(t)+w(t)$ where the function $w$ solves
\begin{multline}\label{1.smooth}
-\gamma\Dt\Dx w-\Dx w+Lw+f(w)-\\-\Nx(\phi'(\Nx w))=h_u(t):=-\Dt^2
u(t)+g+Lu(t),\ \ w\big|_{t=0}=0,
\end{multline}
where $L$ is a sufficiently large number which will be fixed below.
Then, the reminder $v(t)$ should solve
\begin{multline}\label{1.decay}
-\gamma\Dt\Dx v-\Dx v+[f(v+w)-f(w)]+L v =\\=
\Nx(\phi'(\Nx v+w)-\phi'(\Nx w)), \ v\big|_{t=0}=w\big|_{t=0}.
\end{multline}
The desired result will follow from  two lemmas formulated below.

\begin{lemma}\label{Lem4.sm} Let the above assumptions hold. Then
the solution $w(t)$ belongs to $H^2$ for every $t\ge0$ and the
following estimate holds:
\begin{equation}\label{4.sm}
\|w(t)\|_{H^2}\le C_L,
\end{equation}
where the constant $C_L$ may depend on $L$, but is independent of the concrete choice of
the trajectory $u$ starting from the absorbing ball $B_R$ and
satisfying \eqref{4.ap}.
\end{lemma}
\begin{proof}
Indeed, due to \eqref{4.ap}, we have the $L^2$-control of the
right-hand side of \eqref{1.smooth}:
\begin{equation}
\int_t^{t+1}\|h_u(s)\|^2_{L^2}\,ds\le C_L.
\end{equation}
This, allows us to verify estimate \eqref{4.sm} just repeating
word by word the proof of Lemma \ref{Lem3.h2}. By this reason, we
leave the proof to the reader.
\end{proof}

\begin{lemma}\label{Lem4.dec} Let the above assumptions hold and
let $L$ be large enough. Then, the $v$ component of the solution
$u$ satisfies the following estimate:
\begin{equation}\label{4.dec}
\|v(t)\|_{H^1}\le \|u(0)\|_{H^1}e^{-\beta t},
\end{equation}
for some positive exponent $\beta$.
\end{lemma}
\begin{proof} Indeed, multiplying equation \eqref{1.decay} by $v$,
integrating over $\Omega$ and using the monotonicity of $\phi'$, we
get
$$
\frac\gamma2\frac d{dt}\|\Nx v(t)\|^2+\|\Nx
v(t)\|^2_{L^2}+(f(v(t)+w(t))-f(w(t)),v(t))+L\|v(t)\|_{L^2}^2\le0.
$$
Using now assumption \eqref{3.int}, we see that
$$
(f(v(t)+w(t))-f(w(t)),v(t))\ge -C\|v(t)\|^2_{L^2}
$$
and, therefore, the Gronwall inequality gives the desired estimate
\eqref{4.dec} if $L\ge C$.
\end{proof}
Combining estimates \eqref{4.dec} and \eqref{4.sm}, we deduce
\eqref{4.simexp} and finish the proof of the theorem.
\end{proof}

As an immediate corollary of the proved theorem, we obtain the
following result on the regularity of the weak global attractor
$\Cal A$.

\begin{corollary}\label{Cor4.atreg} Let the assumptions of Theorem
\ref{Th1.uniq} hold. Then, the weak global attractor $\Cal A$ of
the solution semigroup $\{S(t),\,t\ge0\}$ associated with equation
\eqref{1.main} belongs to  the space $\Cal E_1$ and is bounded
there:
\begin{equation}
\|\Cal A\|_{\Cal E_1}\le C.
\end{equation}
In particular, the attractor $\Cal A$ is generated by strong
solutions of problem \eqref{1.main}.
\end{corollary}

We are now ready to overcome the difficulty mentioned in Remark
\ref{Rem2.noexp} and to construct an exponential attractor for the
solution semigroup $\{S(t),\, t\ge0\}$. We first recall the
exponential attractor definition adapted to our situation.

\begin{definition}\label{Def4.exat} A set $\Cal M$ is a (weak)
exponential attractor for  the solution semigroup $\{S(t),\, t\ge0\}$
acting in the energy space $\Cal E$ if
\par
1) The set $\Cal M$ is bounded in $\Cal E$ and compact in
$\Cal E_{-1}$;
\par
2) It is semi-invariant: $S(t)\Cal M\subset\Cal M$;
\par
3) It attracts exponentially the images of all bounded in $\Cal E$
sets in the topology of $\Cal E_{-1}$, i.e., there exists a
positive constant $\gamma$ and a monotone function $Q$ such that
$$
\dist_{\Cal E_{-1}}(S(t)B,\Cal M)\le Q(\|B\|_{\Cal E})e^{-\gamma
t},
$$
for all bounded subsets $B\subset\Cal E$;
\par
4) It has finite fractal dimension in $\Cal E_{-1}$.
\end{definition}

\begin{theorem}\label{Th4.aexp} Let the assumptions of Theorem
\ref{Th1.uniq} hold. Then, the solution semigroup
$\{S(t),\,t\ge0\}$ associated with equation \eqref{1.main}
possesses an exponential attractor $\Cal M$ in the sense of
Definition \ref{Def4.exat}. Moreover, this exponential attractor
is bounded in $\Cal E_1$:
$$
\|\Cal M\|_{\Cal E_1}\le C.
$$
\end{theorem}
\begin{proof} Indeed, as we have already mentioned in Remark \ref{Rem2.noexp},
the squeezing property \eqref{2.Ltr} is sufficient to construct an
exponential attractor $\Cal M_d\subset\Cal B_R$ for the discrete semigroup $S(L):\Cal E\to\Cal
E$, see \cite{EFNT}. However, in order to extend it to the attractor
$\Cal M$ of the semigroup with continuous time by the standard
relation
\begin{equation}\label{4.standard}
\Cal M=[\cup_{t\in[0,L]}S(t)\Cal M_d]_{\Cal E_{-1}},
\end{equation}
we need the uniform H\"older continuity of $S(t)\xi$ on both
variables $t$ and $\xi$ on $\Cal M_d$. Although the continuity in $\xi$ follows from
Theorem \ref{Th1.uniq}, we however do not have
the H\"older continuity in time if $\Cal M_d$ is only a subset of
$\Cal E$. So, a priori, more smooth exponential attractor $\Cal
M_d$ for the discrete semigroup $S(nL)$ is required.
\par
We overcome this difficulty by considering our solution semigroup
in a stronger phase space $\Cal E_1$. Then, due to the dissipative estimate
\eqref{1.strongdis}, we can construct a bounded and invariant absorbing set
$\tilde{\Cal B}_R\subset\Cal E_1$ for this semigroup. Furthermore,
repeating the arguments given in the proof of Theorem
\ref{Th2.fd}, but replacing the absorbing set $\Cal B_R\subset\Cal
E$ by the new absorbing set $\tilde{\Cal B}_R\subset\Cal E_1$, we
construct an exponential attractor $\Cal M_d\subset\tilde{\Cal
B}_R\subset\Cal E_1$ for the discrete semigroup $\{S(nL),n\in\Bbb
N\}$ acting in a stronger phase space $\Cal E_1$ (recall that the
attraction property is still formulated
in the topology of $\Cal E_{-1}$).
\par
Since $\Cal M_d$ is now bounded in the stronger space $\Cal E_1$,
a straightforward interpolation gives the desired H\"older
continuity
$$
\|S(t+h)\xi-S(t)\xi\|_{\Cal E_{-1}}\le C|h|^{1/2},
$$
where the constant $C$ is independent of $t$ and $\xi\in\Cal M_d$.
Thus, we can use the standard formula \eqref{4.standard} in order
to construct an exponential attractor $\Cal M\subset\Cal E_1$ for the continuous
semigroup $S(t)$ acting in $\Cal E_1$.
\par
So, we have constructed the exponential attractor $\Cal M$ which
is bounded in $\Cal E_1$ and attracts bounded {\it in the stronger space $\Cal E_1$} sets
in the topology of $\Cal E_{-1}$. Finally, in order to verify that
this attractor will attract exponentially bounded in $\Cal E$ sets
as well, it is sufficient to use \eqref{4.exp} and \eqref{1.lip}
together with the transitivity of exponential attraction, see
\cite{FGMZ}. Theorem \ref{Th4.aexp} is proved.
\end{proof}
\begin{remark}\label{Rem.problem} Theorems \ref{Th4.exp} and
\ref{Th4.aexp} show that any weak energy solution of equation
\eqref{1.main} possesses an asymptotic smoothing property and
belongs to $\Cal E_1$ up to the exponentially decaying terms.
However, the  dissipativity of, say, {\it
classical} solutions remain an open problem in the 3D case even
for
globally bounded nonlinearity $\phi$. Indeed, in order to "freeze" the coefficients and use the
maximal regularity of the linearized equation and
standard localization technique, we need to control the
$C^\alpha$-norm of $\Nx u$ and the $\Cal E_1$-energy gives only
$\Nx u\in L^6$. By this reason, the additional higher energy
estimates are necessary.
\par
The usual way to obtain such estimates (which works perfectly for
the case of semi-linear equations) is to multiply the equations by
$\Dt(\Dx u)+\eb\Dx^2 u$ or something similar. Unfortunately, in
the quasi-linear case this procedure produces the additional term like
$\phi'''(\Nx u)|D^2_x u|^2 D^3_x u$ which, in turn, produces the term
$\|D^2_xu\|^4_{L^4}$ (even in the simplest case of periodic BC and
bounded $\phi$). In the 2D case, this term can be properly
estimated by the interpolation
$$
\|D^2_x u\|^4_{L^4}\le C\|\Dx u\|^2_{L^2}\|\Nx\Dx u\|^2_{L^2}\le
C\|\xi_u\|_{\Cal E_1}^2\|\xi_u\|_{\Cal E_2}^2
$$
(where $\Cal E_2:= W^{3,2}(\Omega)\times W^{2,2}(\Omega)$)
and this allows to verify the existence of classical solutions (in
a fact, even in the case of arbitrary polynomial growth of $\phi$;
using slightly more accurate estimates, see \cite{En2}).
\par
In contrast to that, in the 3D case, we have
$$
\|D^2_x u\|^4_{L^4}\le
C\|\xi_u\|_{\Cal E_1}^{1}\|\xi_u\|^{3}_{\Cal E_2}
$$
which is not strong enough to establish the global existence of the
$\Cal E_2$-solutions.
\par
We also mention that the problem of further regularity of
solutions of \eqref{1.main} is closely related with the analogous
problem for the simplified pseudoparabolic equation
\begin{equation}\label{simple}
\gamma \Dt\Dx u+\Dx u+\Nx(\phi'(\Nx u))=h(t),\ \
u\big|_{\partial\Omega}=0
\end{equation}
for the appropriate external force $h(t)$. However, we do not know
how to establish the global existence and dissipativity of more
regular (than the $H^2$-ones) solutions for that equation.
\par
As we have already mentioned in the introduction, there is an
alternative way to construct more regular solutions in the
particular case of globally bounded $\phi''$ based on the
integro-differential form \eqref{0.funny} and perturbation
arguments. However, the growing in time bounds for more regular
norms seem unavoidable under that method.
\end{remark}

We finish this section by considering the semi-linear case
$\phi\equiv0$ (or more general $\phi''(v)\equiv const$).
Of course, in that case, the problems mentioned in the previous
remark do not appear and the factual regularity of the attractor
is restricted only by the regularity of the data (in particular,
if $\Omega$, $f$ and $g$ are $C^\infty$, the attractor will belong
to $C^\infty$ as well). Moreover,
in that
case, we are able to obtain slightly more strong result on the attraction property for weak energy solutions, namely,
 that the constructed global
($\Cal A$) and exponential ($\Cal M$) attractors of the solution
semigroup $S(t)$  acting in the energy phase space $\Cal E$
attract bounded sets in $\Cal E$ in the topology of the phase
space $\Cal E$ (and not only in a weaker topology of $\Cal E_{-1}$ or $\tilde{\Cal E}$).
This can be obtained from the following refining of Theorem
\ref{Th4.exp}.

\begin{proposition}\label{Prop4.semilin} Let the assumptions of
Theorem \ref{Th1.uniq} be satisfied and assume, in addition, that $\phi\equiv0$ (i.e.,
equation \eqref{1.main} is semi-linear). Then the exponential
attraction \eqref{4.exp} holds in the topology of the phase space
$\Cal E:=[H^1_0\cap L^{q+2}]\times L^2$, i.e., there exists a
positive constant $\gamma$ and a monotone function $Q$ such that
\begin{equation}
\dist_{\Cal E}(S(t)B,\{\xi\in\Cal E_1, \|\xi\|_{\Cal E_1}\le
K\})\le Q(\|B\|_{\Cal E})e^{-\gamma t},
\end{equation}
for every bounded set $B$ in $\Cal E$. Here $K$ is a sufficiently
large number and $\Cal E_1:=[H^2\times H^1_0]\times H^1_0$.
\end{proposition}
\begin{proof}Indeed, exactly as in the proof of Theorem
\ref{Th4.exp}, it is sufficient to verify that
\begin{equation}\label{4.siexp}
\dist_{H^1\cap L^{q+2}}(u(t),\{u_0\in H^2\cap H^1_0,\ \|u_0\|_{H^2}\le K\})
\le C_Re^{-\gamma t},
\end{equation}
for all trajectories $u$ starting from the absorbing
set $\Cal B_R$. In turns, in order to prove \eqref{4.siexp}, it
is sufficient to improve Lemma \ref{Lem4.dec} by replacing the
space $H^1$ in estimate \eqref{4.dec} by $H^1\cap L^{q+2}$.
Namely, we only need to prove that
\begin{equation}\label{4.deci}
\|v(t)\|_{L^{q+2}}\le C_R\|v(0)\|_{H^1}e^{-\gamma t}.
\end{equation}
To this end, we first note that, in the semi-linear case,
the solution $w(t)$ of problem \eqref{1.smooth} has the additional
regularity, namely,
\begin{equation}\label{4.dtsm}
\int_t^{t+1}\|\Dt w(s)\|^2_{H^2}\,ds\le C_R,
\end{equation}
where $C_R$ is independent of $t$. In order to see that, it is
sufficient to express $\Dt\Dx u$ from  equation \eqref{1.smooth}, multiply it by $\Dt \Dx w$
and note that the term $\Dx u$ and  the nonlinearity is under the control due to
the embedding $H^2\to C$ and estimate \eqref{4.sm} and the term
$h_u(t)$ is controlled by \eqref{4.ap}.
\par
Let us now multiply equation \eqref{1.decay} by $\Dt v$ and
use the following formula
$$
(f(w+v)-f(w),\Dt v)=\frac d{dt}[F(v+w)-F(w)-f(w)v]-(\Dt
w,f(w+v)-f(w)-f'(w)v).
$$
where $F(u):=\int_0^uf(v)\,dv$.
Then, we get
\begin{equation}\label{4.simple}
\gamma\|\Dt\Nx v\|^2_{L^2}+\frac d{dt}[(\Phi(v,w),1)+1/2\|\Nx v\|_{L^2}^2+L/2\| v\|^2_{L^2}]=(\Dt
w,R(v,w)),
\end{equation}
where
$$
\Phi(v,w):=F(v+w)-F(w)-f(w)v,\ \ R(v,w):=f(v+w)-f(w)-f'(w)v.
$$
We now recall the following estimates
\begin{equation}\label{4.rf}
\begin{cases}
1.\ |R(v,w)|+|\Phi(v,w)|\le C|v|^2+C_1[f(v+w)-f(w)].v,\\
2.\  |R(v,w)|\le
C_1|v|^2+C_2\Phi(v,w),
\end{cases}
\end{equation}
for some positive constants $C_1$ and $C_2$ which are independent of
$v$ and $w$ (these estimates can be easily derived from our
assumptions \eqref{3.int} on the function $f$, see \cite{Z1}). Using
estimate \eqref{4.rf}(2), together with the embedding $H^2\subset
C$, we transform \eqref{4.simple} as follows:
\begin{multline*}
\frac d{dt}[(\Phi(v,w),1)+1/2\|\Nx v\|_{L^2}^2+L/2\|
v\|^2_{L^2}]+\\+
C\|\Dt w\|_{H^2}[(\Phi(v,w),1)+1/2\|\Nx v\|_{L^2}^2+L/2\| v\|^2_{L^2}]
\le C_1\|\Dt w\|_{H^2}\|v\|_{H^1}^2.
\end{multline*}
Finally, multiplying the last inequality by $t-T$, denoting
$$
M(v,w):=[(\Phi(v,w),1)+1/2\|\Nx v\|_{L^2}^2+L/2\| v\|^2_{L^2}]
$$
(note that $M\ge0$ due to our choice of $L$)
and using \eqref{4.rf}(1), we have
\begin{multline*}
\frac{d}{dt}[(t-T) M(v,w)]+C\|\Dt w\|_{H^2}[(t-T) M(v,w)]\le\\\le  C_1(t-T+1)\|\Dt
w\|_{H^2}\|v\|_{H^1}^2+ C_2(f(v+w)-f(v),v)
\end{multline*}
which together with the Gronwall inequality gives
\begin{multline}\label{4.smsm}
M(v(T+1),w(T+1))\le
C\|v\|^2_{L^\infty([T,T+1],H^1)}+\\+C\int_T^{T+1}|(f(v(t)+w(t))-f(w(t)),v(t))|\,dt.
\end{multline}
Recall that the first term in the right-hand side of
\eqref{4.smsm} decays exponentially due to \eqref{4.dec} and, in
order to see that the second one also decays exponentially, it is
sufficient to multiply equation \eqref{1.decay} by $v$, integrate
over $[T,T+1]\times\Omega$ and use \eqref{1.decay} again. Thus, we
have proved that
$$
M(v(t),w(t))\le C_R\|v(0)\|^2_{H^1}e^{-\beta t}.
$$
Or which is the same (due to \eqref{4.dec} again),
\begin{equation}\label{4.Fdec}
(\Phi(v(t),w(t)),1)\le C_R\|v(0)\|^2_{H^1}e^{-\beta t}.
\end{equation}
In order to deduce \eqref{4.deci} from \eqref{4.Fdec}, it only
remains to note that our assumptions on the non-linearity $f$ give
the following inequality
$$
\Phi(v,w)\ge \kappa|v|^{q+2}-C|v|^2
$$
for some positive constants $\kappa$ and $C$ depending on the proper norms of $w$, see \cite{Z1}. Thus,
estimate \eqref{4.deci} is verified and Proposition
\ref{Prop4.semilin} is proved.
\end{proof}

\begin{corollary}\label{Cor4.last} Under the assumptions of
Proposition \ref{Prop4.semilin}, the global attractor $\Cal A$
(resp. the exponential attractor $\Cal M$) constructed above,
attracts (resp. attracts exponentially) bounded sets in $\Cal E$
in the topology of the space $\Cal E$ as well.
\end{corollary}

\begin{remark}\label{Rem.semilinear} To the best
 of our knowledge, when $\Omega \in\R^3$ the growth restriction $q\le 4$ has been
always posed in order to study the energy solutions. As it follows
from our result, this growth restriction is unnecessary and we
have well-posedness, dissipativity and asymptotic regularity of
weak energy solutions without any restrictions on the growth
exponent $q$.
\end{remark}

\section{Related equations and concluding remarks}\label{s6}
Although only the case of quasi-linear strongly damped wave
equations in the form of \eqref{1.main} has been considered in the
paper, the developed methods have  general nature and can be
successfully
applied to many related problems. Some of them are briefly
discussed below. More detailed exposition of that and similar problems
will be given somewhere else.

\subsection{Strongly Damped Kirchhoff Equation}
We start with the so-called  Kirchhof equation with strongly damping
term which can be considered as a simplified version of the
quasi-linear wave equation \eqref{1.main}:
\begin{equation}
\begin{cases}\label{Kir}
\Dt^2u-\gamma\Dt\Dx u-\Dx u- \Phi(\|\Nx u\|^2_{L^2})\Dx u+f(u)=g, \\
u\big|_{\partial\Omega}=0, \ \ \xi_u(0):=(u(0),\Dt u(0))=\xi_0,
\end{cases}
\end{equation}
where $f$ satisfies \eqref{3.int} and $\Phi \in C^1(
\R^{+}\rightarrow \R^{+})$ satisfies the condition
$$
\Phi(s)s-\int_0^s\Phi(\tau)d\tau \geq 0, \ \ \ \forall s \in \R^{+}.
$$
Indeed, the non-linearity $\Phi(\|\Nx u\|_{L^2}^2)\Dx u=\Nx(\Phi(\|\Nx
u\|^2_{L^2})\Nx u)$ is milder  than the one $\Nx (\phi'(\Nx u)\Nx
u)$ considered before and, by this reason, all the above results
can be obtained (analogously, but with great simplifications) for
the Kirchhof equation \eqref{Kir} as well. In particular, arguing
as in Theorem \ref{Th1.uniq}, we obtain the uniqueness of energy
solutions for any growth exponent $q$ (for the non-linearity $f$).
To the best of our knowledge, this result was known before
for $q\le 5$ only.

\subsection{Membrane Equation} Our next application is the
so-called quasilinear strongly damped membrane equation which can
be considered as a natural 4th order (5th order, being pedantic)
analog of problem \eqref{1.main}:
\begin{equation}
\begin{cases}\label{2.beam}
\Dt^2u+\Dx \phi(\Dx u)+ \Dx^2 u +\gamma\Dt\Dx^2 u+f(u)=g, \\
u\big|_{\partial\Omega}=\Dx u\big|_{\partial\Omega}=0, \ \
\xi_u(0):=(u(0),\Dt u(0))=\xi_0,
\end{cases}
\end{equation}
where $\phi$ and $f$ are given non-linearities, $g\in L^2(\Omega)$ is a given
source function. We assume that the function $\phi\in C^2$ is monotone and satisfies the
analogue  of assumptions \eqref{1.int}
\begin{equation}\label{7.int}
 \kappa_2|u|^{p-1}\le\phi'(u)\le \kappa_1|u|^{p-1},
\end{equation}
for some positive $\kappa_i$ and some $p\ge1$ and the
non-linearity $f\in C^1$ satisfies the standard dissipativity assumption
\begin{equation}\label{71.int}
f(u)u\ge -C.
\end{equation}
Analogously the case of equation \eqref{1.main}, it is natural to introduce the weak ($\Cal E$)
and the strong energy ($\Cal E_1$) spaces via
\begin{equation}\label{7.energy}
\Cal E:=[W^{2,p+1}(\Omega)\cap H^1_0(\Omega)]\times L^2(\Omega)
\end{equation}
and
\begin{equation}\label{7.sten}
\Cal E_1:=\{\xi_u\in [H^3_\Delta\cup W^{2,p+1}(\Omega)]\times H^2_\Delta,\ \phi(\Dx u)\in
L^2(\Omega)\}\equiv [H^3_\Delta\cap W^{2,2p}(\Omega)]\times
H^2_\Delta,
\end{equation}
where $H^s_\Delta:=D((-\Delta)^{s/2})$ in $L^2(\Omega)$.
Here we have used that the condition $\phi(\Dx u)\in L^2$ is
equivalent to $\Dx u\in L^{2p}$).
\par
The energy functional now reads
$$
E(\xi_u):=1/2\|\Dx u\|^2_{L^2}+(\Phi(\Dx u),1)+1/2\|\Dt
u\|^2_{L^2}+(F(u),1)-(g,u),
$$
where $\Phi(z):=\int_0^z\phi(v)\,dv$ and $F(z):=\int_0^zf(v)\,dv$
and the standard energy estimate (multiplication of the equation
by $\Dt u+\eb u$) gives
\begin{equation}\label{7.dis}
\|\xi_u(t)\|_{\Cal E}^2+\int_0^t\|\Dt u(s)\|^2_{H^2}\,ds\le
Q(\|\xi_u(0)\|_{\Cal E}^2)e^{-\alpha t}+Q(\|g\|_{L^2})
\end{equation}
for some positive $\alpha$ and monotone $Q$.
\par
Thus, in contrast to the case of equation \eqref{1.main}, the
$C$-norm of the solution $u(t)$ is controlled by the energy norm.
By this reason, not any additional growth assumptions for $f$ is
now necessary.
\par
Another essential simplification appears in the proof of
the uniqueness of energy solutions. Indeed, following the proof of
Theorem \eqref{Th1.uniq}, we now need to multiply the equation for
the differences of two solutions by $\Dx^{-2}\Dt v$ (instead of
$\Dt(-\Dx)^{-1} v$, $v:=u_1-u_2$) and the only non-trivial term to
estimate will be
$$
(\phi(\Dx u_1)-\phi(\Dx u_2),(-\Dx)^{-1}\Dt v).
$$
Since $H^2\subset C$ and the $L^2$-norm of the $\Dt v$ is under
the control, this term can be estimated analogously to
\eqref{1.fest} {\it without} the restrictions on the growth
exponent $p$. Thus, the existence and uniqueness theorem for
energy solutions now holds for arbitrary growth of the
non-linearity $\phi$.
\par
In addition, denoting $v=\Dt u$, $\Cal E_{-1}:=H^{-2}_\Delta\times H^2_\Delta$
and repeating word by word the proof of propositions
\ref{Prop1.dtreg} and \ref{Prop1.dtsmooth}, we establish the
analogue of estimate \eqref{1.dtest} and the smoothing property:
\begin{multline}\label{7.sm}
\|\Dt u(t)\|_{H^2}^2+\|\Dt^2 u(t)\|_{H^{-2}}^2+\int_t^{t+1}\|\Dt^2
u(s)\|^2_{L^2}\,ds\le\\\le \frac{1+t^N}{t^N}\(Q(\|\xi_u(0)\|_{\Cal
E})e^{-\alpha t}+Q(\|g\|_L^2)\).
\end{multline}
Expressing now the term $\phi(\Dx)$ from equation \eqref{2.beam}
and using \eqref{7.dis} and \eqref{7.sm}, we obtain the control of
the $L^2$-norm of $\phi(\Dx u)$ or which is equivalent, the
control of the $W^{2,2p}$-norm of $u$. Thus, the semigroup
$S(t):\Cal E\to\Cal E$ associated with problem \eqref{2.beam}
possesses the partial smoothing property of the form:
\begin{equation}\label{7.fun}
S(t):[W^{2,p+1}(\Omega)\cap H^1_0(\Omega)]\times L^2(\Omega)\to
 [W^{2,2p}(\Omega)\cap H^1_0(\Omega)]\times [H^2(\Omega)\times
 H^1_0(\Omega)].
\end{equation}
The finite-dimensional global attractor $\Cal A$ ($(\Cal E,\Cal E_{-1})$-attractor,
to be more precise) can be now obtained repeating word by word the
proof of Theorem \ref{Th2.fd}. However in contrast to the case of equation \eqref{1.main}, the smoothing
\eqref{7.fun} together with the interpolation
$$
\|u\|_{W^{2,p+1}}\le
C\|u\|_{W^{2,2}}^{1/(p+1)}\|u\|^{p/(p+1)}_{W^{2,2p}}
$$
allows to verify that this attractor is finite-dimensional in
$\Cal E$ and attracts in the initial topology of the phase spce $\Cal E$ as well, see
Remark \ref{Rem2.batr}. Thus, we have verified the following
result.

\begin{theorem} \label{Th7.weak} Let the above assumptions hold.
Then, equation \eqref{2.beam} is well-posed and dissipative in the
phase space $\Cal E$ and the associated semigroup $S(t)$ possesses
a finite-dimensional global attractor $\Cal A$ in $\Cal E$.
\end{theorem}

Let us now discuss the further regularity of
solutions/attractor of problem \eqref{2.beam}. We first note that
the analogue of the main result of Section \ref{s2} (strong
solutions) is now immediate since the non-linear term arising
after the multiplication of \eqref{2.beam} by $\Dx u$ can be
estimated via
$$
(\Dx \phi(\Dx u),\Dx u)=-(\phi'(\Dx u)\Nx\Dx u,\Nx\Dx u)\le0
$$
and, in contrast to Section \ref{s2}, we do not have here any
problems with boundary terms. By that reason, we need not neither
the localization technique of Lemma \ref{3.h2} nor the
approximations of $\phi$ by the bounded non-linearities (see the
proof of Theorem \ref{Th3.strong}) and the existence of a strong
solution can be verified by the usual Galerkin method.
\par
Finally, again in contrast to the case of equation \eqref{1.main},
the question about the classical solutions of \eqref{2.beam} can
be solved in a relatively simple way. Indeed, let $\theta:=\Dx u$.
Then this function solves the equation
\begin{equation}\label{7.simple}
\gamma\Dt\theta+\theta+\phi(\theta)=h_u(t):=(-\Dx)^{-1}\Dt^2
u+(-\Dx)^{-1}f(u)-(-\Dx)^{-1}g.
\end{equation}
Since $\phi$ is monotone and without loss of generality we  have  $h_u\in L^2([0,T],
C^\alpha(\Omega))$ (due to \eqref{7.sm} and embedding $H^2\subset
C$), then the usual maximum/comparison principle allows to verify
the dissipative estimate for $\theta$ in $C^\alpha(\Omega)$. Thus,
we will have the asymptotic smoothing property to $u(t)\in
C^{2+\alpha}(\Omega)$ on the attractor (in  fact, in the superlinear case
$p>1$, we even have that $\Dx u(t)\in L^\infty(\Omega)$ in a
finite time). As we have already mentioned in Remark
\ref{Rem.problem}, the regularity $\Dx u(t)\in C^\alpha(\Omega)$
is crucial for the application of the standard localization
technique and the existence of classical solutions. Thus, we have
the following result.

\begin{theorem}\label{Th7.smooth} Let the domain $\Omega$ and the
data $\phi$, $f$ and $g$ be smooth enough and satisfy
\eqref{7.int} and \eqref{71.int}. Then, the attractor $\Cal A$ of
problem \eqref{2.beam} consists of classical solutions. In
addition, if the above data is of class $C^\infty$, then the
attractor $\Cal A$ belongs to $[C^\infty(\Omega)]^2$ as well.
\end{theorem}

To the best of our knowledge, the existence and uniqueness of energy solutions to the problem
\eqref{2.beam}, was known only if the  nonlinearity $\phi(\cdot)$ has a linear growth, see \cite{BGS} where
the abstract differential equation of the form \eqref{abstr2} is
considered (and the assumptions on the linear operators $A_1$, $A_2$ and $N$ admit equation \eqref{2.beam}  with a sub-linear $\phi$
as a particular case. The attractor theory for that abstract equation has been developed in
\cite{Pi}, but under the additional assumption that $A_2^{-1/2}N$
is a compact operator which automatically {\it excludes} the case
of equation \eqref{2.beam} from consideration (even in the case of sub-linear $\phi$). In contrast to that, the
methods developed in our paper allow to give the {\it complete}
theory of equation \eqref{2.beam} (existence, uniqueness,
regularity, classical solutions, attractors, etc.) for arbitrary
growth rate of $\phi$.

\subsection{The Wave Equation with Structural Damping} To
continue, we note that our technique allows to improve essentially
of the so-called wave
equation with structural damping
\begin{equation}
\begin{cases}\label{6.beam}
\Dt^2u- \Dx u +\gamma(-\Dx)^{\alpha} \Dt u+f(u)=g, \\
u\big|_{\partial\Omega}=0, \ \ \xi_u(0):=(u(0),\Dt u(0))=\xi_0,
\end{cases}
\end{equation}
where $\alpha \in [1/2,1]$, $g\in L^2(\Omega)$ and  the function $f$
satisfies \eqref{3.int} for some exponent $q$. Indeed, following the scheme of Theorem \eqref{Th1.uniq}, for the
uniqueness of energy solutions, we need to multiply the equation
for difference between two solutions $u_1$ and $u_2$ by $v+\eb
(-\Dx)^{-\alpha}\Dt v$ and estimate the term
\begin{equation}\label{7.f}
(f(u_1)-f(u_2),(-\Dx)^{-\alpha}\Dt v)
\end{equation}
using that $(f(u_1)-f(u_2),v)\ge -K$ and that the $L^2$-norm of $\Dt v$ is
under the control. In the case $\alpha>3/4$, we have the embedding
$H^{2\alpha}\subset C$ and the term \eqref{7.f} can be estimated
exactly as in \eqref{1.fest}. A little more accurate analysis
shows that the same estimate still works for $\alpha=3/4$ without
any restrictions on the growth exponent $q$, but for $\alpha<3/4$
an additional growth restriction
\begin{equation}\label{7.gr}
q+2<\frac{6}{3-4\alpha},\ \ \ \alpha\in(1/2,3/4)
\end{equation}
is required (in  fact, this technique works also for
$\alpha\le1/2$, but, in this case, the standard scheme of proving
the uniqueness is preferable; by this reason, we discuss here the
case $\alpha>1/2$ only).
\par
Thus, analogously to Theorems \ref{Th1.uniq} and \ref{Th2.fd}, we
have the following result
\begin{theorem}\label{Th7.a} Let  the assumption
\eqref{3.int} be valid with the exponent $q$ satisfying \eqref{7.gr} (if
$\alpha<3/4$). Then, problem \eqref{6.beam} is well-posed and
dissipative in the energy space $\Cal E:=[H^1_0\cap L^{q+2}]\times
L^2$ and possesses the finite-dimensional global attractor $\Cal
A$ (in the sense of the definition given in Section \ref{s.atr}
with $\Cal E_{-1}:=H^{1-\alpha}_\Delta\times H^{\alpha}_\Delta$).
\end{theorem}
Furthermore, the only multiplication of equation \eqref{6.beam} by $\Dx
u$ is {\it insufficient} for proving the existence of strong
solutions if $\alpha\ne1$ since the additional terms $\|\Dt u\|_{H^1}^2$
and $\Dt(\Nx\Dt u,\Nx u)$ need to be estimated. However, in order to
overcome this difficulty,
we just need to prove the analogue of Theorems \ref{Prop1.dtreg}
and \eqref{Prop1.dtsmooth} (smoothing property for $\Dt u$) before.
Indeed, this smoothing property gives the control of the
$L^2([0,T],H^1)$ and $C([0,T],H^{1-\alpha})$-norms of $u$ and the
$C([0,T],H^{\alpha})$-norm of $\Dt u$ and that is {\it exactly}
what we need in order to be able to estimate the above mentioned
terms.
\par
Thus, the multiplication by $\Dx u$ finally works and we can prove
that the attractor $\Cal A$ consists of strong solutions belonging
to $\Cal E_1:=H^{\alpha}_\Delta\times H^{1+\alpha}_\Delta$. Since
$\alpha>1/2$, we have $H^{1+\alpha}\subset C$ and the further
regularity of the attractor (e.g., classical or
$C^\infty$-solution) can be obtained from the linear theory).
In addition, the analogue of Proposition \ref{Prop4.semilin} also
can be proved repeating word by word the arguments of Section
\ref{s4}. Thus, the attractor $\Cal A$ is finite-dimensional and
attracts the images of bounded sets in the topology of the initial
phase space $\Cal E$ as well.

\begin{theorem}\label{Th7.last} Let the assumptions of Theorem
\ref{Th7.a} hold and let $\Omega$, $f$ and $g$ be smooth enough.
 Then, the attractor $\Cal A$ consists of classical solutions and
 attracts the images of bounded sets in $\Cal E$ in the topology
 of the initial phase space $\Cal E$.
 \end{theorem}

To the best of our knowledge, the above mentioned results were
known before only under the growth restriction $q\le 4$, see
\cite{CCh} (which has been considered as a critical growth exponent
for that equation).

\end{document}